\documentclass[a4paper,12pt]{article}
\usepackage[top=2cm, bottom=2cm, outer=2cm, inner=2cm, headsep=14pt]{geometry}
\usepackage{amsmath,amsfonts}
\usepackage{enumitem}
\usepackage{adjustbox}
\usepackage{color}  %%need for bbulet

\usepackage{tikz} %%need for figure
\usepackage{url} %%need for url

\allowdisplaybreaks

\newcommand{\G}{\Gamma}

\def\0{{\boldsymbol 0}}

\def\RR{{\mathbb R}}

\def\R{{\mathcal R}}
\def\P{{\mathcal P}}
\def\B{{\mathcal B}}

\def\ol{\overline}
\def\ds{\displaystyle}

\def\Lra{\Leftrightarrow}
\def\Ra{\Rightarrow}

\newtheorem{theorem}{Theorem}[section]
\newtheorem{lemma}[theorem]{Lemma}
\newtheorem{corollary}[theorem]{Corollary}
\newtheorem{proposition}[theorem]{Proposition}
\newtheorem{definition}[theorem]{Definition}

\newtheorem{notation}[theorem]{Notation}
\newtheorem{remark}[theorem]{Remark}

\newtheorem{example}[theorem]{Example}

\newtheorem{problem}{Problem}[section]

\definecolor{ForestGreen}{RGB}{12, 110, 46}
\definecolor{ForestGreenTwo}{RGB}{120, 110, 86}

\newenvironment{proof}{{\noindent\it Proof. }}{\nopagebreak\hspace*{0.5cm}\hfill$\hbox{\rule{3pt}{6pt}}$\smallskip}

\definecolor{ForestGreen}{RGB}{12, 110, 46}
\definecolor{ForestGreenTwo}{RGB}{120, 110, 86}

\newfont\fiverm{cmr5}
\def\eeq{\end{equation}} 
\def\lbeq#1{\begin{equation} \label{#1}} 
 
\title{
On (almost) $2$-$Y$-homogeneous distance-biregular graphs
}

\author{
{Blas Fern\'andez}\\
{\small UP IAM}\\
{\small University of Primorska}\\
{\small Muzejski trg 2, 6000 Koper, Slovenia }\\
{\small Blas.Fernandez@famnit.upr.si} \and
{Safet Penji\'c}\\
{\small UP IAM}\\
{\small University of Primorska}\\
{\small Muzejski trg 2, 6000 Koper, Slovenia }\\
{\small Safet.Penjic@iam.upr.si}
}

\begin{document}

\maketitle

\begin{abstract}
Let $\G$ denote a bipartite graph with vertex set $X$, color partitions $Y$, $Y'$,  and assume that every vertex in $Y$ has eccentricity $D\ge 3$. For $z\in X$ and non-negative integer $i$, let $\G_{i}(z)$ denote the set of vertices in $X$ that are at distance $i$ from $z$. Graph $\G$ is {\it almost $2$-$Y$-homogeneous} whenever for all $i \; (1\leq i \leq D-2)$ and for all $x\in Y$, $y \in \G_2(x)$ and $z \in \G_{i}(x)\cap\G_i(y)$, the number of common neighbours of $x$ and $y$ which are at distance $i-1$ from $z$ is independent of the choice of $x$, $y$ and $z$. In addition, if the above condition holds also for $i=D-1$, then we say that $\G$ is {\it $2$-$Y$-homogeneous}.

Now, let $\G$ denote a distance-biregular graph. In this paper we study the intersection arrays of $\G$ and we give sufficient and necessary conditions under which $\G$ is (almost) $2$-$Y$-homogeneous. In the case when $\G$ is $2$-$Y$-homogeneous we write the intersection numbers of the color class $Y$ in terms of three parameters.
\end{abstract}

%%\hfill{\bf confidential; please do not distribute}

\smallskip
{\small
%% 	05C75  	Structural characterization of families of graphs
%% I am using MSC2010 database
%% 05C50  	Graphs and linear algebra (matrices, eigenvalues, etc.)
%% https://mathscinet.ams.org/mathscinet/msc/msc2010.html?t=05Cxx&btn=Current
%% 05C12  	Distance in graphs
%% https://mathscinet.ams.org/mathscinet/msc/msc2010.html?t=05Exx&btn=Current
%% 05C62  	Graph representations (geometric and intersection representations, etc.)
%% 05E30  	Association schemes, strongly regular graphs

}

%%%%%%%%%%%%%%%%%%%%%%%%%%%%%%%%%%%%%%%%%%%%%%%%%%%%%%%%%%%%%%%%%%%%%%%%
%%\tableofcontents % Inhalt hier

%%%%%%%%%%%%%%%%%%%%%%%%%%%%%%%%%%%%%%%%%%%%%%%%%%%%%%%%%%%%%%%%%%%%%%%%
\section{Introduction}
\label{1a}

The reader is referred to Section~\ref{pa} for formal definitions.

Let $\G$ denote a graph  with the vertex set $X$. By $\partial(x,y)$ we denote the usual distance between vertices $x$ and $y$. For $z\in X$ and an integer $i$, let $\G_i(z)$ denote the set of all vertices which are at distance $i$ from $z$. A vertex $x\in X$ is said to be {\it distance-regularized} if for each $y\in X$, the numbers $c_i(x):=|\G_{i-1}(x)\cap\G_1(y)|$, $a_i(x):=|\G_{i}(x)\cap\G_1(y)|$ and $b_i(x):=|\G_{i+1}(x)\cap\G_1(y)|$ depend only on the distance $\partial(x,y)=i$, and are independent of the choice of $y\in\G_{i}(x)$. For the moment assume that $x$ is a distance-regularized vertex of $\G$. Then the array
$$
(b_0(x),b_1(x),\ldots,b_d(x);1,c_2(x),\ldots,c_{d}(x))
$$
\noindent
is called the {\it intersection array for $x$} (or, {\it intersection array of $\G$ with respect to $x$}). A connected graph in which every vertex is distance-regularized is called a {\it distance-regularized} graph. A special case of such graphs are {\it distance-regular} graphs in which all vertices have the same intersection array. Other examples are bipartite graphs in which vertices in the same color class have the same intersection array, but which are not distance-regular. We call these graphs {\it distance-biregular}. {\sc Godsil} and {\sc Shawe-Taylor} \cite{GST} proved that every distance-regularized graph is either distance-regular or distance-biregular. The family of distance-biregular graphs is quite rich, and examples can be obtained  from several algebraic and geometric objects \cite[Section~4]{DC}. Moreover, these objects can be seen as non-symmetric association schemes \cite[Section~1]{DC}. There are a number of excellent articles on distance-biregular graphs, e.g., \cite{Fp, Fs, GST, MST, NK, SH}.

In this paper we study the combinatorial structure of distance-biregular graphs. We give an answer to the following problem.

\begin{problem}
Let $\G$ denote a distance-biregular graph with vertex set $X=Y\cup Y'$, color partitions $Y$, $Y'$, and let $D$ denote the eccentricity of vertices in $Y$. For vertices $x,y\in X$, let $\G_{i,j}(x,y)=\G_i(x)\cap\G_j(y)$. {\em Find necessary and sufficient conditions on the intersection arrays of $\G$} for which the graph has one of the following two combinatorial structures:
\begin{enumerate}[rightmargin=1.8cm, leftmargin=2cm,label=\rm(\alph*)]
\item for all $i \; (1\leq i \leq D-2)$, and for all $x\in Y$, $y\in\G_2(x)$ and $z\in\G_{i,i}(x,y)$ the number $|\G_{i-1}(z)\cap \G_{1,1}(x,y)|$ is independent of the choice of $x$, $y$ and $z$. If $\G$ has this combinatorial structure, then we say that $\G$ is {\em almost $2$-$Y$-homogeneous}.
\item for all $i \; (1\leq i \leq D-1)$ and for all $x\in Y$, $y\in\G_2(x)$ and $z\in\G_{i,i}(x,y)$ the number $|\G_{i-1}(z)\cap \G_{1,1}(x,y)|$ is independent of the choice of $x$, $y$ and $z$. If $\G$ has this combinatorial structure, then we say that $\G$ is {\em $2$-$Y$-homogeneous}.
\end{enumerate}
%%Find all possible types for intersection array of one color class, and write them with respect to two (or three) parameters.
\end{problem}

%%Next, we give motivations for study these two family of graphs. 

This paper is motivated by a desire to find a combinatorial characterization of bipartite graphs $\G=(Y\cup Y',\R)$ which are (almost) $2$-$Y$-homogeneous and to classify all such graphs. If $\G$ is distance-regular, then this situation occurs if and only if the intersection array of $\G$ is at least one of the types given in \cite[Theorem~1.2]{NKsm}. For a general graph $\G$ the above problem is very difficult, and in this paper we study distance-biregular graphs. One of our long term goals is to study and classify all distance-biregular graphs which have a unique irreducible $T$-module (up to isomorphism) of endpoint $2$, which is thin, and this paper is one of the steps for solving this problem. See also \cite{FMe} for the results about the structure of irreducible $T$-module, when the endpoint is $1$, also for the case when $\G$ is distance-biregular. Study of thin $T$-modules of endpoint $2$ of distance-regular graphs is in the last several years very active area of research, see, for example \cite{FM, MM1, MM2, MM4, MM3, MmTp}. The second long term goal for better understanding combinatorial properties of such graphs, is {\em finding an algorithm which will draw a such graph (i.e., explicitly give all edges between vertices) from its intersection arrays} (if such graph exists) - the problem that because of its difficulty does not have enough attention inside the  mathematical community. A ``simple'' combinatorial structure of the graph can contribute to more easily finding such an algorithm, and a combinatorial structure of the $2$-$Y$-homogeneous distance-biregular graphs is promising. Let's mention also that {\sc Mohar} and {\sc Shawe-Taylor} \cite{MST} showed that $\G$ is a distance-biregular graph with vertices of valency $2$ if and only if $\G$ is either a complete bipartite graph $K_{2,n}$ for some $n\geq1$, or the subdivision graph of a $(\kappa, g)$-graph (see \cite[Corollary~3.5]{MST} and Subsection~\ref{kH} for more details). It turns out that if $\G$ is a $(\kappa,g)$-cage graph with vertex set $X$ then the subdivision graph $S(\G)$ is a $2$-$X$-homogeneous distance-biregular graph (see Section~\ref{D2}). Our main result is the following theorem.

\begin{theorem}
\label{rH}
Let $\G$ denote a $2$-$Y$-homogeneous distance-biregular graph with color classes $Y$, $Y'$, and let $D$ denote the eccentricity of vertices in the color class $Y$. Assume that $D\ge 3$, and let $c_2'$ denote the intersection number $|\G_1(x)\cap\G_{1}(y)|$ $(x\in Y',\,y\in\G_{2}(x))$ of the color class $Y'$.
\begin{enumerate}[label={\rm(\roman*)}]
\item If $c_2'=1$ then the intersection array of the color class $Y$ is one of the following two types
\begin{align*}
(k,k'-1,k-1,k'-1,k-1,\ldots,k'-1,k-1;1,1,1,1,\ldots,1,k'),\qquad \mbox{ for odd } D,\\
(k,k'-1,k-1,k'-1,k-1,\ldots,k-1,k'-1;1,1,1,1,\ldots,1,k),\qquad \mbox{ for even } D.
\end{align*}
\item If $c_2'=2$ then the first three intersection numbers of the color class $Y$ are
$$
(k,\frac{k-1}{c-1},k-c,\frac{k-1}{c-1}-c,\ldots;1,c,c+1,\ldots)
$$
for some integers $k$ and $c\ge 2$.
\item If $c_2'\ge 3$ then $D\le 5$. Moreover, we have
\begin{enumerate}[label={\rm(\alph*)}]
\item If $D=3$ then the intersection array of the color class $Y$ is of the form $(k,c,k-c;1,c,c+1)$ for some integers $k$ and $c$, where $k>c\ge 2$. 
\item If $D=4$ then the intersection array of the color class $Y$ is of the form 
$
(k,k'-1,k-c, k'-1-\frac{c(c'-1)}{\gamma}; 1, c, \frac{c(c'-1)}{\gamma}+1, k)
$
for some positive integers $k$, $k'$ and $c$, where $k>c\ge 2$, $k'>2$, $c'=\frac{(k'-1)(c-1)}{k-1}+1$ and $\gamma=\frac{(c-1)(c'-2)}{k'-2}+1$.
\item If $D=5$ then the intersection array of the color class $Y$ is of the form 
$
(k,k'-1,k-c, b_3, b_4; 1, c, c_3, c_4, k')
$
for some positive integers $k$, $k'$ and $c$, where $k>c\ge 2$, $k'>2$, 
$c_3=k'-1-\frac{c(c'-1)}{\gamma}$, 
$c'=\frac{(k'-1)(c-1)}{k-1}+1$, 
$\gamma=\frac{(c-1)(c'-2)}{k'-2}+1$, 
$b_3=k'-c_3$, 
$c_4=\frac{k(k'-1)-\frac{c(b_3-1)(k-1)}{c-1}}{c_3}$ and $b_4=k-c_4$.
\end{enumerate}
\end{enumerate}
\end{theorem}
By \cite[Lemma~2.3]{GST} the intersection numbers of the color class $Y'$ can be computed in terms of the intersection numbers of the color class $Y$.

In the end of this section we summarize our main results. In Section~\ref{pa} we give notation, formal definitions and we recall some basic properties of distance-biregular graphs. In Section~\ref{La} we define certain scalars $\gamma_i$ and we compute some equalities in the case when $c_2'\ge 2$, which we use later. We also show that $c_2'\ge 3$ implies $D\le 5$. In Section~\ref{D2} we show that the subdivision graph $S(\Gamma)$ of a $(\kappa,g)$-cage graph $\Gamma=(X,\R)$ is $2$-$X$-homogeneous. In Sections~\ref{D1} and \ref{Ev} we define and study scalars $\Delta_i$ $(2\le i\le \min\{D-1,D'-1\})$ (where $D'$ denotes the eccentricity of vertices in the color class $Y$). These scalars can be computed from the intersection array of a given distance-biregular graph and they play an important role: from their values we can determine if a given distance-biregular graph is (almost) $2$-$Y$-homogeneous or not. We show that $\Delta_i\ge 0$ and that the following (i)--(iv) are equivalent: (i) The scalar $\Delta_i=0$; (ii) For all $x\in Y$, $y\in\G_2(x)$ and $z\in\G_{i,i}(y,z)$, the number $|\G_{1}(x)\cap\G_1(y)\cap\G_{i-1}(z)|$ is independent of the choice of $z$; (iii) There exist $x \in Y$ and $z\in\G_i(x)$ such that for all $y\in\G_{2,i}(x,z)$ the number $|\G_{1}(x)\cap\G_1(y)\cap\G_{i-1}(z)|$ is independent of the choice of $y$; (iv) There exist $x \in Y$ and $y\in\G_2(x)$ such that for all $z\in\G_{i,i}(x,y)$ the number $|\G_{1}(x)\cap\G_1(y)\cap\G_{i-1}(z)|$ is independent of the choice of $z$. As a corollary we get (1) $\G$ is $2$-$Y$-homogeneous if and only if $\Delta_i=0$ $(2\le i\le\min\{D-1,D'-1\})$; and (2) $\G$ is almost $2$-$Y$-homogeneous if and only if $\Delta_i=0$ $(2\le i\le D-2)$. In Section~\ref{oH} we prove that the scalars $\gamma_i$ are nonzero, when $k'\ge 3$ and $D\ge 3$. In Section~\ref{Eb} we study distance-biregular graphs with $c_2'=1$ and we show that $\G$ is almost $2$-$Y$-homogeneous with $c_2=1$ if and only if $c_i=1$ for every integer $i$ $(1\le i\le D-1)$. In Sections~\ref{Fc} and \ref{Fe} we give possible types for the intersection array for a $2$-$Y$-homogeneous distance-biregular graph, when $D\in\{3,4,5\}$. The proof of Theorem~\ref{rH} is in Section~\ref{Fe}. We finish the paper giving simple examples and some open problems which could be of interest for further research.

%%%%%%%%%%%%%%%%%%%%%%%%%%%%%%%%%%%%%%%%%%%%%%%%%%
%%%%%%%%%%%%%%%%%%%%%%%%%%%%%%%%%%%%%%%%%%%%%%%%%%
%%%%%%%%%%%%%%%%%%%%%%%%%%%%%%%%%%%%%%%%%%%%%%%%%%
%%%%%%%%%%%%%%%%%%%%%%%%%%%%%%%%%%%%%%%%%%%%%%%%%%
%%%%%%%%%%%%%%%%%%%%%%%%%%%%%%%%%%%%%%%%%%%%%%%%%%

\section{Preliminaries}
\label{pa}

An {\it(undirected) graph} $\G$ is a pair $(X, \R)$, where $X$ is a nonempty set and $\R$ is a collection of one or two element subsets of $X$. The elements of $X$ are called the {\it vertices} of $\G$, and the elements of $\R$ are called the {\it edges} of $\G$. A one element subset of $X$ in $\R$ is an edge which starts and ends at the same vertex - it is called a {\it loop}. When $xy\in \R$ $(x\ne y)$, we say that vertices $x$ and $y$ are {\it adjacent}, or that $x$ and $y$ are {\it neighbors}. Adjacency between vertices $x$ and $y$ will be denoted by $x\sim y$. A graph is {\it finite} if both its vertex set and edge set are finite. A graph is {\it simple} if it has no loops and no two of its edges join the same pair of vertices.

Let $\G = (X,\R)$ be a graph. For any two vertices $x, y \in X$, a {\it walk} of length $h$ from $x$ to $y$ is a sequence $[x_0,x_1,x_2,\ldots,x_h]$ $(x_i\in X,\, 0\le i\le h)$ such that $x_0 = x$, $x_h = y$, and $x_i$ is adjacent to $x_{i+1}$ $(0\le i\le h-1)$. We say that $\G$ is {\it connected} if for any $x, y\in X$, there is a walk from $x$ to $y$. A {\it path} is a walk such that all vertices of the walk are distinct. From now on, assume that $\G$ is finite, simple and connected.

For any $x, y\in X$, the {\it distance} between $x$ and $y$, denoted $\partial(x, y)$, is the length of a shortest walk from $x$ to $y$. The {\it diameter} $d = d(\G)$ is defined to be 
$$
d = \max\{\partial(u,v)\,|\,u, v\in X\}.
$$
%%A walk in $\G$ is said to be {\it closed} if it starts and ends at the same vertex. 
The {\it eccentricity} of $x$, denoted by $\varepsilon=\varepsilon(x)$, is the maximum distance between $x$ and any other vertex of $\G$. Note that the diameter of $\G$ equals $\max\{\varepsilon(x)\mid x\in X\}$.

A simple graph in which each pair of distinct vertices is joined by an edge is called a {\it complete graph}. The complete graph on $n$ vertices is denoted by $K_n$. A {\it bipartite} (or {\it $(Y,Y')$-bipartite}) graph is one whose vertex set can be partitioned into two subsets $Y$ and $Y'$, so that each edge has one endpoint in $Y$ and another endpoint in $Y'$. The vertex sets $Y$ and $Y'$ in such a partition is called a {\it color partition} (or {\it bipartition}) of the graph. A {\it complete bipartite graph} is a simple bipartite graph with color partitions $Y$ and $Y'$ in which each vertex of $Y$ is joined to each vertex of $Y'$; if $|Y| = m$ and $|Y'|= n$, such a graph is denoted by $K_{m,n}$. A graph $\G$ is {\it regular} with valency $k$ if each vertex in $\G$ has exactly $k$ neighbours.

\subsection{Distance-regularized vertex and distance-regular graph}

The concept of distance appears in all of science, as in our daily lives. In the study of graphs, distances play an important role, and maybe one of the main reasons for it is in their wide applicability. The distance between two vertices in a graph is surprisingly simple and useful notion. It has led to the definition of several graph parameters such as the diameter, the radius, the average distance, the eccentricity of a vertex, the adjacency matrix, the distance-$i$ matrix, as well as to the definition of several graphs classes (distance-balanced, strongly distance-balanced, nicely distance-balance, distance-regular, distance-biregular, distance-transitive, distance-hereditary). Some interesting articles regarding different uses of distances are \cite{BC, EMJS, FKR, KTKR, MRRA, WMS, NP, RO}.

Let $\G$ denote a graph with vertex set $X$ and diameter $d$. For a vertex $x\in X$ and any non-negative integer $i$ not exceeding $d$, let $\G_i(x)$ denote the subset of vertices in $X$ that are at distance $i$ from $x$. Let $\G_{-1}(x) = \G_{d+1}(x) := \emptyset$. For any two vertices $x$ and $y$ in $X$ at distance $i$, let $C_i(x,y):=\G_{i-1}(x)\cap\G_1(y)$, $A_i(x,y):=\G_{i}(x)\cap\G_1(y)$, $B_i(x,y):=\G_{i+1}(x)\cap\G_1(y)$. We say that a vertex $x \in X$ is {\it distance-regularized} (or that a graph $\G$ is {\it distance-regular around $x$}) if the numbers $|A_i(x,y)|$,  $|B_i(x,y)|$ and $|C_i(x,y)|$ do not depend on the choice of $y \in \G_i(x) \; (0 \le i \le d)$; in this case, the numbers  $|A_i(x,y)|$,  $|B_i(x,y)|$ and $|C_i(x,y)|$ are simply denoted  by $a_i(x)$,  $b_i(x)$ and $c_i(x)$ respectively, and are called the {\it intersection numbers of $x$}. A graph $\G$ is called {\it distance-regular} if there are integers $b_i$, $c_i$ $(0\le i\le d)$ which satisfy $c_i = |C_i(x, y)|$ and $b_i = |B_i (x, y)|$ for any two vertices $x$ and $y$ in $X$ at distance $i$. Clearly such a graph is regular of valency $k := b_0$ and  $a_i := |A_i(x,y)| = k - b_i - c_i$ $(0\le i\le d)$ is the number of neighbours of $y$ in $\G_i(x)$ for $x, y\in X$ $(\partial(x,y)=i)$. Note that a graph is distance-regular if each of its vertices is distance-regularized, and if all of its vertices have the same intersection numbers. For more information about distance-regular graphs we refer the reader to \cite{NB, BN, BCN, DKT, FMA, SP}.

\subsection{Equitable partition and intersection diagram}

A {\it partition} of a graph $\G$ is a collection $\{\P_1, \P_2, \dots, \P_s\}$ of  nonempty subsets of $X$, such that $X=\ds\bigcup_{i=1}^s \P_i$ and $\P_i\cap\P_j=\emptyset$ for all $i,j$ $(1 \le i,j \le s,\, i\ne j)$. An {\it equitable partition} of a graph $\G$ is a partition $\{\P_1, \P_2, \dots, \P_s\}$ of its vertex set, such that for all integers $i,j$ $(1 \le i,j \le s)$ the number $c_{ij}$ of neighbours, which a vertex in the cell $\P_i$ has in the cell $\P_j$, is independent of the choice of the vertex in $\P_i$. 

Assume that $\G$ is a graph of diameter $d$, and pick arbitrary vertices $x$ and $y\in\G_h(x)$, for some fixed $h$ $(0\le h\le \varepsilon(x))$. The collection of all nonempty subsets $\G_{i,j}(x,y)$ $(0\le i,j\le d)$ is a partition of the vertex set of $\G$ which is called the {\em intersection diagram of $\G$ of rank $h$}.

\subsection{Distance-biregular graph (DBG)}

A natural continuation in understanding combinatorial and algebraic properties of distance-regular graphs is to study the family of graphs known as distance-biregular graphs. Let $\G$ denote a graph (not necessary regular) with the property that each vertex is distance-regularized (i.e., $\G$ is distance-regular around every vertex). Such a graph is said to be {\it distance-regularized}. A distance-regularized graph is called {\it distance-biregular} if the following (i)--(iii) hold: (i) it is bipartite; (ii) the vertices in the same color partition have the same intersection array; and (iii) the vertices in different color classes have different intersection arrays. A well-known result by {\sc Godsil} and {\sc Shawe-Taylor} \cite{GST} states that if $\G$ is a connected graph that is distance-regular around every vertex, then $\G$ is distance-regular or distance-biregular. The complete bipartite graph $K_{m,n}$ is an example of distance-biregular graph. The subdivision graph of the Petersen graph is an example of $2$-$Y$-homogeneous distance-biregular graph (here $Y$ is the set of vertices of the Petersen graph). We refer to \cite{ DC, FM, Fp, Fs, MST, NK} for further research on distance-biregular graphs.

Let $\G$ denote a $(Y,Y')$-distance-biregular graph with vertex set $X$. Pick $x \in X$ and let $\varepsilon(x)$ denote the eccentricity of $x$. Since $\G$ is bipartite, we have $a_i(x)=0$ for $0 \le i \le \varepsilon(x)$. Note that all vertices from $Y$ ($Y'$, respectively) have the same eccentricity. We denote this common eccentricity by $D$ ($D'$, respectively). Observe that $|D-D'| \le 1$ and the diameter of $\G$ equals $\max\{D,D'\}$. In addition, all vertices from $Y$ ($Y'$, respectively) have the same valency $k$ ($k'$, respectively). Observe that for $0 \leq i \leq D$, $c_i+b_i=k$ if $i$ is even, while $c_i+b_i=k'$ if $i$ is odd.

Note that, if $D=2$ then $\G$ is a complete bipartite graph with $D'\in \{1,2\}$ (if $D'=3$, then for any $x\in Y'$, $\G_3(x)\subseteq Y$, which yields $D\ge 3$, a contradiction). Moreover, if $D=2$ then $\G$ is $2$-$Y$-homogeneous by definition. For the rest of the paper we assume that $D\ge 3$ (which also yields $D'\geq 3$ - just use the same argument as in the previous sentence). From now on, until the end of the paper, we use the following notation.

\begin{notation}\label{GN}{\rm
Let $\G$ denote a $(Y,Y')$-distance-biregular graph with vertex set $X$, and color partitions $Y$ and $Y'$. Let $D$ ($D'$, respectively) denote the eccentricity of vertices from $Y$ ($Y'$, respectively), and assume that $D\ge 3$. For vertices $x_1,x_2,\hdots,x_k\in X$ and non-negative integers $i_1,i_2,\hdots,i_k$ $(0\le i_1,i_2,\hdots,i_k \le d)$ we define
$
\G_{i_1,i_2,\hdots,i_k}(x_1,x_2,\hdots,x_k)=\bigcap_{\ell=1}^k \G_{i_\ell}(x_\ell).
$
For $x\in Y$ and $z\in\G_i(x)$ $(0< i< D)$ we abbreviate $c_0=0$, $c_i=|\G_{i-1,1}(x,z)|$, $b_i=|\G_{i+1,1}(x,z)|$, $b_{D}=0$, $k_i:=|\G_i(x)|$,
and write 
$
(k,b_1,\ldots,b_D;1,c_2,\ldots,c_D)
$
for the intersection array of the color partition $Y$. For $u\in Y'$ and $v\in\G_i(u)$ $(0< i< D')$ we abbreviate $c'_0=0$, $c'_i=|\G_{i-1,1}(u,v)|$, $b'_i=|\G_{i+1,1}(u,v)|$, $b'_{D'}=0$, $k'_i:=|\G_i(u)|$ and write 
$
(k',b'_1,\ldots,b'_D;1,c'_2,\ldots,c'_{D'})
$
for the intersection array of the color partition $Y'$.
}\end{notation}

Note that, for $0 \leq i \leq D'$, $c_i'+b_i'=k'$ if $i$ is even, while $c_i'+b_i'=k$ if $i$ is odd (see Figure~\ref{01}). Using the same intersection diagram, for any $x\in Y$ and $u\in Y'$, we also get
\begin{equation}
\label{Ej}
|\G_i(x)|=\frac{b_0b_1\cdots b_{i-1}}{c_1c_2\cdots c_i}~(1\le i\le D),
\qquad
|\G_i(u)|=\frac{b'_0b'_1\cdots b'_{i-1}}{c'_1c'_2\cdots c'_i}~(1\le i\le D').
\end{equation}

{\small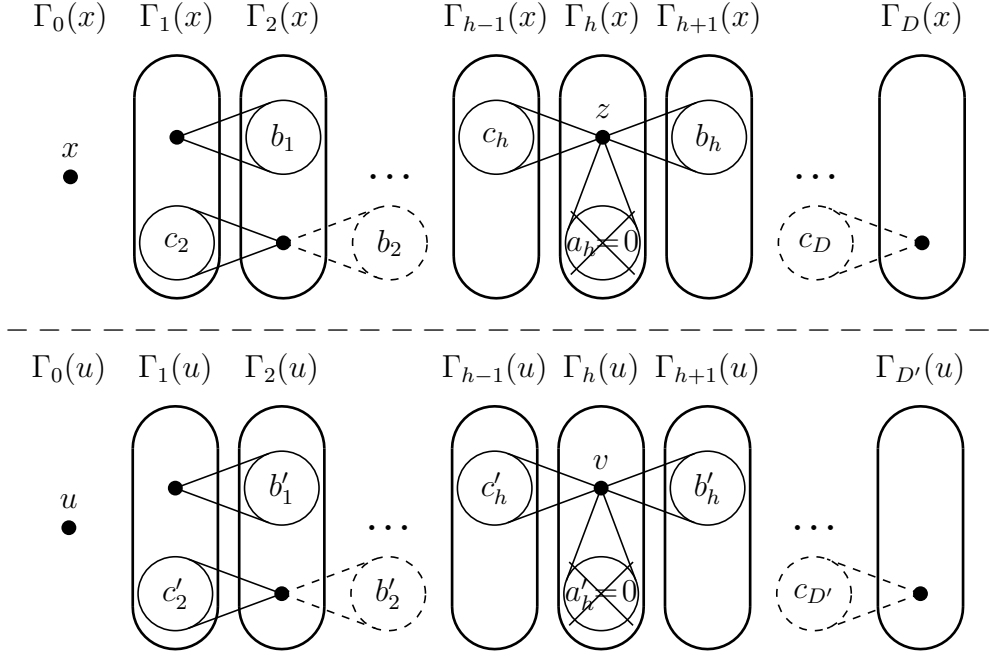
\begin{figure}[t]{\rm
\begin{center}
\begin{tikzpicture}[scale=0.7]
\draw [line width=1pt] (1.2,5)-- (1.2,2);
\draw [line width=1pt] (2.8,5)-- (2.8,2);
\draw [shift={(2,5)},line width=1pt]  plot[domain=0:3.141592653589793,variable=\t]({1*0.8*cos(\t r)+0*0.8*sin(\t r)},{0*0.8*cos(\t r)+1*0.8*sin(\t r)});
\draw [shift={(2,2)},line width=1pt]  plot[domain=3.141592653589793:6.283185307179586,variable=\t]({1*0.8*cos(\t r)+0*0.8*sin(\t r)},{0*0.8*cos(\t r)+1*0.8*sin(\t r)});
\draw [shift={(4,5)},line width=1pt]  plot[domain=0:3.141592653589793,variable=\t]({1*0.8*cos(\t r)+0*0.8*sin(\t r)},{0*0.8*cos(\t r)+1*0.8*sin(\t r)});
\draw [shift={(4,2)},line width=1pt]  plot[domain=3.141592653589793:6.283185307179586,variable=\t]({1*0.8*cos(\t r)+0*0.8*sin(\t r)},{0*0.8*cos(\t r)+1*0.8*sin(\t r)});
\draw [line width=1pt] (3.2,5)-- (3.2,2);
\draw [line width=1pt] (4.8,5)-- (4.8,2);
\draw [line width=1pt] (7.2,5)-- (7.2,2);
\draw [line width=1pt] (8.8,5)-- (8.8,2);
\draw [line width=1pt] (9.2,5)-- (9.2,2);
\draw [line width=1pt] (10.8,5)-- (10.8,2);
\draw [line width=1pt] (11.2,5)-- (11.2,2);
\draw [line width=1pt] (12.8,5)-- (12.8,2);
\draw [line width=1pt] (15.2,5)-- (15.2,2);
\draw [line width=1pt] (16.8,5)-- (16.8,2);
\draw [shift={(8,5)},line width=1pt]  plot[domain=0:3.141592653589793,variable=\t]({1*0.8*cos(\t r)+0*0.8*sin(\t r)},{0*0.8*cos(\t r)+1*0.8*sin(\t r)});
\draw [shift={(10,5)},line width=1pt]  plot[domain=0:3.141592653589793,variable=\t]({1*0.8*cos(\t r)+0*0.8*sin(\t r)},{0*0.8*cos(\t r)+1*0.8*sin(\t r)});
\draw [shift={(12,5)},line width=1pt]  plot[domain=0:3.141592653589793,variable=\t]({1*0.8*cos(\t r)+0*0.8*sin(\t r)},{0*0.8*cos(\t r)+1*0.8*sin(\t r)});
\draw [shift={(16,5)},line width=1pt]  plot[domain=0:3.141592653589793,variable=\t]({1*0.8*cos(\t r)+0*0.8*sin(\t r)},{0*0.8*cos(\t r)+1*0.8*sin(\t r)});
\draw [shift={(16,2)},line width=1pt]  plot[domain=3.141592653589793:6.283185307179586,variable=\t]({1*0.8*cos(\t r)+0*0.8*sin(\t r)},{0*0.8*cos(\t r)+1*0.8*sin(\t r)});
\draw [shift={(12,2)},line width=1pt]  plot[domain=3.141592653589793:6.283185307179586,variable=\t]({1*0.8*cos(\t r)+0*0.8*sin(\t r)},{0*0.8*cos(\t r)+1*0.8*sin(\t r)});
\draw [shift={(10,2)},line width=1pt]  plot[domain=3.141592653589793:6.283185307179586,variable=\t]({1*0.8*cos(\t r)+0*0.8*sin(\t r)},{0*0.8*cos(\t r)+1*0.8*sin(\t r)});
\draw [shift={(8,2)},line width=1pt]  plot[domain=3.141592653589793:6.283185307179586,variable=\t]({1*0.8*cos(\t r)+0*0.8*sin(\t r)},{0*0.8*cos(\t r)+1*0.8*sin(\t r)});
\draw [line width=.6pt] (8,4.25) circle (0.7cm);
\draw [line width=.6pt] (12,4.25) circle (0.7cm);
\draw [line width=.6pt] (10,2.25) circle (0.7cm);
\draw [line width=.6pt] (10,4.25)-- (8.245,4.905724789831837);
\draw [line width=.6pt] (10,4.25)-- (8.245,3.594275210168163);
\draw [line width=.6pt] (10,4.25)-- (11.755,4.905724789831837);
\draw [line width=.6pt] (10,4.25)-- (11.755,3.5942752101681634);
\draw [line width=.6pt] (10,4.25)-- (9.344275210168169,2.495);
\draw [line width=.6pt] (10,4.25)-- (10.655724789831831,2.495);
\node at (6,3.5) {$\boldsymbol\ldots$};
\node at (14,3.5) {$\boldsymbol\ldots$};
\node at (0,6.5) {$\G_0(x)$};
\node at (2,6.5) {$\G_1(x)$};
\node at (4,6.5) {$\G_2(x)$};
\node at (8,6.5) {$\G_{h-1}(x)$};
\node at (10,6.5) {$\G_{h}(x)$};
\node at (12,6.5) {$\G_{h+1}(x)$};
\node at (16,6.5) {$\G_{D}(x)$};
\node at (0,4) {$x$};
\fill (0,3.5) circle [radius=0.14];
\node at (8,4.25) {$c_h$};
\node at (12,4.25) {$b_h$};
\node at (10,2.25) {$a_h\!\!=\!0$};
\draw [line width=.6pt] (9.4,1.65)-- (10.6,2.85);
\draw [line width=.6pt] (9.4,2.85)-- (10.6,1.65);
\node at (10,4.75) {$z$};
\fill (10,4.25) circle [radius=0.14];

\fill (2,4.25) circle [radius=0.14];
%%\draw [line width=.6pt] (0,4.25) circle (0.7cm);
\draw [line width=.6pt] (4,4.25) circle (0.7cm);
%%\draw [line width=.6pt] (2,2.25) circle (0.7cm);
\draw [line width=.6pt] (2,4.25)-- (3.755,4.905724789831837);
\draw [line width=.6pt] (2,4.25)-- (3.755,3.5942752101681634);
%%\draw [line width=.6pt] (2,4.25)-- (1.344275210168169,2.495);
%%\draw [line width=.6pt] (2,4.25)-- (2.655724789831831,2.495);
\node at (4,4.25) {$b_1$};

%%\draw [line width=.6pt] (8,4.25) circle (0.7cm);
\draw [line width=.6pt] (2,2.25) circle (0.7cm);
\fill (4,2.25) circle [radius=0.14];
\draw [line width=.6pt] (4,2.25)-- (2.245,2.905724789831837);
\draw [line width=.6pt] (4,2.25)-- (2.245,1.594275210168163);
\node at (2,2.25) {$c_2$};
\draw [line width=.6pt, dashed] (6,2.25) circle (0.7cm);
\draw [line width=.6pt, dashed] (4,2.25)-- (5.755,2.905724789831837);
\draw [line width=.6pt, dashed] (4,2.25)-- (5.755,1.5942752101681634);
\node at (6,2.25) {$b_2$};

\draw [line width=.6pt, dashed] (14,2.25) circle (0.7cm);
\fill (16,2.25) circle [radius=0.14];
\draw [line width=.6pt, dashed] (16,2.25)-- (14.245,2.905724789831837);
\draw [line width=.6pt, dashed] (16,2.25)-- (14.245,1.594275210168163);
\node at (14,2.25) {$c_D$};
\end{tikzpicture}
$-------------------------------$
\begin{tikzpicture}[scale=0.7]
\draw [line width=1pt] (1.2,5)-- (1.2,2);
\draw [line width=1pt] (2.8,5)-- (2.8,2);
\draw [shift={(2,5)},line width=1pt]  plot[domain=0:3.141592653589793,variable=\t]({1*0.8*cos(\t r)+0*0.8*sin(\t r)},{0*0.8*cos(\t r)+1*0.8*sin(\t r)});
\draw [shift={(2,2)},line width=1pt]  plot[domain=3.141592653589793:6.283185307179586,variable=\t]({1*0.8*cos(\t r)+0*0.8*sin(\t r)},{0*0.8*cos(\t r)+1*0.8*sin(\t r)});
\draw [shift={(4,5)},line width=1pt]  plot[domain=0:3.141592653589793,variable=\t]({1*0.8*cos(\t r)+0*0.8*sin(\t r)},{0*0.8*cos(\t r)+1*0.8*sin(\t r)});
\draw [shift={(4,2)},line width=1pt]  plot[domain=3.141592653589793:6.283185307179586,variable=\t]({1*0.8*cos(\t r)+0*0.8*sin(\t r)},{0*0.8*cos(\t r)+1*0.8*sin(\t r)});
\draw [line width=1pt] (3.2,5)-- (3.2,2);
\draw [line width=1pt] (4.8,5)-- (4.8,2);
\draw [line width=1pt] (7.2,5)-- (7.2,2);
\draw [line width=1pt] (8.8,5)-- (8.8,2);
\draw [line width=1pt] (9.2,5)-- (9.2,2);
\draw [line width=1pt] (10.8,5)-- (10.8,2);
\draw [line width=1pt] (11.2,5)-- (11.2,2);
\draw [line width=1pt] (12.8,5)-- (12.8,2);
\draw [line width=1pt] (15.2,5)-- (15.2,2);
\draw [line width=1pt] (16.8,5)-- (16.8,2);
\draw [shift={(8,5)},line width=1pt]  plot[domain=0:3.141592653589793,variable=\t]({1*0.8*cos(\t r)+0*0.8*sin(\t r)},{0*0.8*cos(\t r)+1*0.8*sin(\t r)});
\draw [shift={(10,5)},line width=1pt]  plot[domain=0:3.141592653589793,variable=\t]({1*0.8*cos(\t r)+0*0.8*sin(\t r)},{0*0.8*cos(\t r)+1*0.8*sin(\t r)});
\draw [shift={(12,5)},line width=1pt]  plot[domain=0:3.141592653589793,variable=\t]({1*0.8*cos(\t r)+0*0.8*sin(\t r)},{0*0.8*cos(\t r)+1*0.8*sin(\t r)});
\draw [shift={(16,5)},line width=1pt]  plot[domain=0:3.141592653589793,variable=\t]({1*0.8*cos(\t r)+0*0.8*sin(\t r)},{0*0.8*cos(\t r)+1*0.8*sin(\t r)});
\draw [shift={(16,2)},line width=1pt]  plot[domain=3.141592653589793:6.283185307179586,variable=\t]({1*0.8*cos(\t r)+0*0.8*sin(\t r)},{0*0.8*cos(\t r)+1*0.8*sin(\t r)});
\draw [shift={(12,2)},line width=1pt]  plot[domain=3.141592653589793:6.283185307179586,variable=\t]({1*0.8*cos(\t r)+0*0.8*sin(\t r)},{0*0.8*cos(\t r)+1*0.8*sin(\t r)});
\draw [shift={(10,2)},line width=1pt]  plot[domain=3.141592653589793:6.283185307179586,variable=\t]({1*0.8*cos(\t r)+0*0.8*sin(\t r)},{0*0.8*cos(\t r)+1*0.8*sin(\t r)});
\draw [shift={(8,2)},line width=1pt]  plot[domain=3.141592653589793:6.283185307179586,variable=\t]({1*0.8*cos(\t r)+0*0.8*sin(\t r)},{0*0.8*cos(\t r)+1*0.8*sin(\t r)});
\draw [line width=.6pt] (8,4.25) circle (0.7cm);
\draw [line width=.6pt] (12,4.25) circle (0.7cm);
\draw [line width=.6pt] (10,2.25) circle (0.7cm);
\draw [line width=.6pt] (10,4.25)-- (8.245,4.905724789831837);
\draw [line width=.6pt] (10,4.25)-- (8.245,3.594275210168163);
\draw [line width=.6pt] (10,4.25)-- (11.755,4.905724789831837);
\draw [line width=.6pt] (10,4.25)-- (11.755,3.5942752101681634);
\draw [line width=.6pt] (10,4.25)-- (9.344275210168169,2.495);
\draw [line width=.6pt] (10,4.25)-- (10.655724789831831,2.495);
\node at (6,3.5) {$\boldsymbol\ldots$};
\node at (14,3.5) {$\boldsymbol\ldots$};
\node at (0,6.5) {$\G_0(u)$};
\node at (2,6.5) {$\G_1(u)$};
\node at (4,6.5) {$\G_2(u)$};
\node at (8,6.5) {$\G_{h-1}(u)$};
\node at (10,6.5) {$\G_{h}(u)$};
\node at (12,6.5) {$\G_{h+1}(u)$};
\node at (16,6.5) {$\G_{D'}(u)$};
\node at (0,4) {$u$};
\fill (0,3.5) circle [radius=0.14];
\node at (8,4.25) {$c_h'$};
\node at (12,4.25) {$b_h'$};
\node at (10,2.25) {$a_h'\!\!=\!0$};
\draw [line width=.6pt] (9.4,1.65)-- (10.6,2.85);
\draw [line width=.6pt] (9.4,2.85)-- (10.6,1.65);
\node at (10,4.75) {$v$};
\fill (10,4.25) circle [radius=0.14];

\fill (2,4.25) circle [radius=0.14];
%%\draw [line width=.6pt] (0,4.25) circle (0.7cm);
\draw [line width=.6pt] (4,4.25) circle (0.7cm);
%%\draw [line width=.6pt] (2,2.25) circle (0.7cm);
\draw [line width=.6pt] (2,4.25)-- (3.755,4.905724789831837);
\draw [line width=.6pt] (2,4.25)-- (3.755,3.5942752101681634);
%%\draw [line width=.6pt] (2,4.25)-- (1.344275210168169,2.495);
%%\draw [line width=.6pt] (2,4.25)-- (2.655724789831831,2.495);
\node at (4,4.25) {$b_1'$};

%%\draw [line width=.6pt] (8,4.25) circle (0.7cm);
\draw [line width=.6pt] (2,2.25) circle (0.7cm);
\fill (4,2.25) circle [radius=0.14];
\draw [line width=.6pt] (4,2.25)-- (2.245,2.905724789831837);
\draw [line width=.6pt] (4,2.25)-- (2.245,1.594275210168163);
\node at (2,2.25) {$c_2'$};

\draw [line width=.6pt, dashed] (6,2.25) circle (0.7cm);
\draw [line width=.6pt, dashed] (4,2.25)-- (5.755,2.905724789831837);
\draw [line width=.6pt, dashed] (4,2.25)-- (5.755,1.5942752101681634);
\node at (6,2.25) {$b_2'$};

\draw [line width=.6pt, dashed] (14,2.25) circle (0.7cm);
\fill (16,2.25) circle [radius=0.14];
\draw [line width=.6pt, dashed] (16,2.25)-- (14.245,2.905724789831837);
\draw [line width=.6pt, dashed] (16,2.25)-- (14.245,1.594275210168163);
\node at (14,2.25) {$c_{D'}$};
\end{tikzpicture}
\caption{\rm 
Intersection diagrams (of rank $0$) of a $(Y,Y')$-distance-biregular graph, with respect to $x\in Y$ and $u\in Y'$, and a graphical representation of the intersection numbers.
}
\label{01}
\end{center}
}\end{figure}}

In Lemma~\ref{4w} we recall some relations between the intersection numbers of a distance-biregular graph. The proof of Lemma~\ref{jH} is an easy exercise. Both of these lemmas will be used later in the paper.

\begin{lemma}[{{\rm\cite[Proposition~2]{DC}}}]
\label{4w}
With reference to Notation~\ref{GN}, the following hold.
\begin{enumerate}[label={\rm(\roman*)}]
\item $c_i'\leq c_{i+1} \; (1\leq i\leq D-1) $ and $c_i\leq c'_{i+1} \; (1\leq i\leq D'-1) $.
\item $b_i\leq b'_{i-1} \; (1\leq i\leq D-1) $ and $b'_i\leq b_{i-1} \; (1\leq i\leq D'-1) $.
\end{enumerate}
\end{lemma}

\begin{lemma}
\label{jH}
With reference to Notation~\ref{GN}, if $i+j\le D$ and $i+j$ is an even number, then $c_i\le b_j$.
\end{lemma}

\begin{proof}
Fix $u\in Y$, $v\in\G_{i}(u)$, $w\in\G_{i+j,j}(u,v)$ and note that $c_i=|\G_{i-1,1}(u,v)|=|\G_{i-1,1,j+1}(u,v,w)|\le |\G_{1,j+1}(v,w)|=b_j$.
\end{proof}

{\small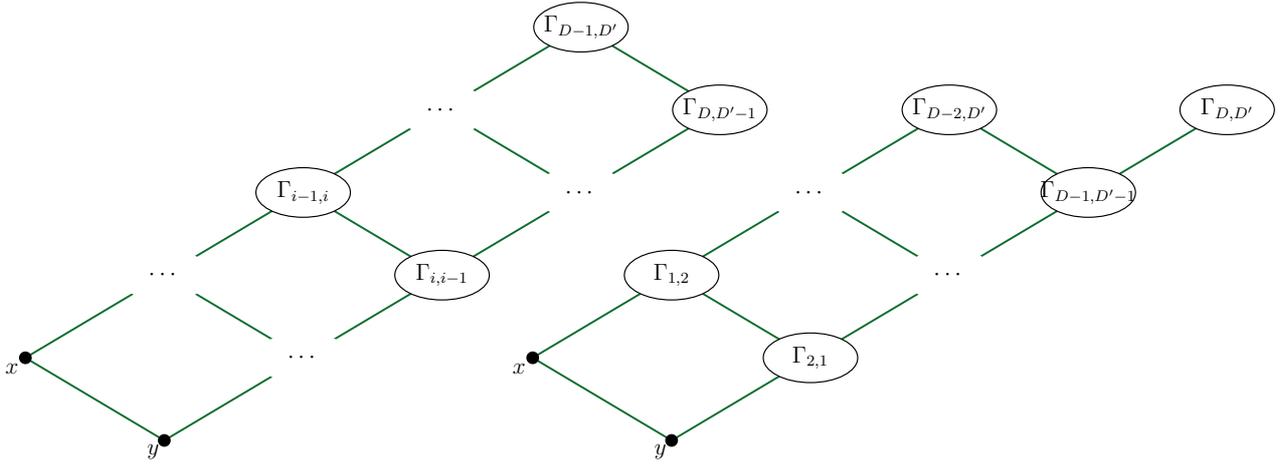
\begin{figure}[t]
{\rm
\begin{center}
\adjustbox{width=\textwidth,keepaspectratio}{
\begin{tikzpicture}[scale=.5]
\draw [line width=1pt, draw=ForestGreen] (-2.,-7.)-- (18.,5.);
\draw [line width=1pt, draw=ForestGreen] (3.,-10.)-- (23.,2.);
\draw [line width=1pt, draw=ForestGreen] (-2.,-7.)-- (3.,-10.);
\draw [line width=1pt, draw=ForestGreen] (3.,-4.)-- (8.,-7.);
\draw [line width=1pt, draw=ForestGreen] (8.,-1.)-- (13.,-4.);
\draw [line width=1pt, draw=ForestGreen] (13.,2.)-- (18.,-1.);
\draw [line width=1pt, draw=ForestGreen] (18.,5.)-- (23.,2.);
%%\draw [line width=1pt, draw=ForestGreen] (3.,-4.)-- (3.,-10.);
%%\draw [line width=1pt, draw=ForestGreen] (8.,-1.)-- (8.,-7.);
%%\draw [line width=1pt, draw=ForestGreen] (13.,2.)-- (13.,-4.);
%%\draw [line width=1pt, draw=ForestGreen] (18.,5.)-- (18.,-1.);
%%\draw [line width=1pt, draw=ForestGreen] (23.,5.)-- (23.,2.);
%%\draw [line width=1pt, draw=ForestGreen] (18.,5.)-- (23.,5.);
%%\draw [line width=1pt, draw=ForestGreen] (13.,2.)-- (23.,2.);
%%\draw [line width=1pt, draw=ForestGreen] (8.,-1.)-- (18.,-1.);
%%\draw [line width=1pt, draw=ForestGreen] (3.,-4.)-- (13.,-4.);
%%\draw [line width=1pt, draw=ForestGreen] (-2.,-7.)-- (8.,-7.);
%%\draw[fill=white, draw=black, line width=0.6pt] (23.,5.) ellipse (1.7cm and .9cm);
\draw[fill=white, draw=black, line width=0.6pt] (18.,5.) ellipse (1.7cm and .9cm);
\draw[fill=white, draw=white, line width=0.6pt] (13.,2.) ellipse (1.7cm and .9cm);
%%\draw[fill=white, draw=white, line width=0.6pt] (18.,2.) ellipse (1.7cm and .9cm);
\draw[fill=white, draw=black, line width=0.6pt] (23.,2.) ellipse (1.7cm and .9cm);
\draw[fill=white, draw=black, line width=0.6pt] (8.,-1.) ellipse (1.7cm and .9cm);
%%\draw[fill=white, draw=black, line width=0.6pt] (13.,-1.) ellipse (1.7cm and .9cm);
\draw[fill=white, draw=white, line width=0.6pt] (18.,-1.) ellipse (1.7cm and .9cm);
%%\draw[fill=white, draw=white, line width=0.6pt] (8.,-4.) ellipse (1.7cm and .9cm);
\draw[fill=white, draw=black, line width=0.6pt] (13.,-4.) ellipse (1.7cm and .9cm);
\draw[fill=white, draw=white, line width=0.6pt] (8.,-7.) ellipse (1.7cm and .9cm);
%%\draw[fill=white, draw=black, line width=0.6pt] (3.,-7.) ellipse (1.7cm and .9cm);
\draw[fill=white, draw=white, line width=0.6pt] (3.,-4.) ellipse (1.7cm and .9cm);
\fill (-2,-7) circle [radius=0.23];
\fill (3,-10) circle [radius=0.23];
\node at (-2.1,-7,1) {\normalsize $x$};
\node at (2.6,-10.4) {\normalsize $y$};
\node at (3,-4.) {\normalsize $\cdots$};
\node at (13,2.) {\normalsize $\cdots$};
%%\node at (3,-7.) {\normalsize $\G_{1,1}$};
\node at (8,-7.) {\normalsize $\cdots$};
\node at (8,-1.) {\normalsize $\G_{i-1,i}$};
%%\node at (8,-4.) {\normalsize $\cdots$};
%%\node at (23,5.) {\normalsize $\G_{D,D}$};
\node at (23,2.) {\normalsize $\G_{D,D'-1}$};
\node at (18,5.) {\normalsize $\G_{D-1,D'}$};;
%%\node at (18,2.) {\normalsize $\cdots$};
%%\node at (13,-1.) {\normalsize $\G_{i,i}$};
\node at (18,-1.) {\normalsize $\cdots$};
\node at (13,-4.) {\normalsize $\G_{i,i-1}$};
\end{tikzpicture}~\hspace{-50mm}~\begin{tikzpicture}[scale=.5]
%%%%%%%%%%%%%%%%%%%%%%%%%%%%%%%%%%%%%%%%%%
%%%%%%%%%%%%%%%%%%%%%%%%%%%%%%%%%%%%%%%%%%
%%%%%%%%%%%%%%%%%%%%%%%%%%%%%%%%%%%%%%%%%%
%%%%%%%%%%%%%%%%%%%%%%%%%%%%%%%%%%%%%%%%%%
%%%%%%%%%%%%%%%%%%%%%%%%%%%%%%%%%%%%%%%%%%
%%\draw [line width=1pt, draw=ForestGreen] (-2.,-7.)-- (18.,5.);
\draw [line width=1pt, draw=ForestGreen] (-2.,-7.)-- (13,2.);
\draw [line width=1pt, draw=ForestGreen] (3.,-10.)-- (23.,2.);
\draw [line width=1pt, draw=ForestGreen] (-2.,-7.)-- (3.,-10.);
\draw [line width=1pt, draw=ForestGreen] (3.,-4.)-- (8.,-7.);
\draw [line width=1pt, draw=ForestGreen] (8.,-1.)-- (13.,-4.);
\draw [line width=1pt, draw=ForestGreen] (13.,2.)-- (18.,-1.);
%%\draw [line width=1pt, draw=ForestGreen] (18.,5.)-- (23.,2.);
%%\draw [line width=1pt, draw=ForestGreen] (3.,-4.)-- (3.,-10.);
%%\draw [line width=1pt, draw=ForestGreen] (8.,-1.)-- (8.,-7.);
%%\draw [line width=1pt, draw=ForestGreen] (13.,2.)-- (13.,-4.);
%%\draw [line width=1pt, draw=ForestGreen] (18.,5.)-- (18.,-1.);
%%\draw [line width=1pt, draw=ForestGreen] (23.,5.)-- (23.,2.);
%%\draw [line width=1pt, draw=ForestGreen] (18.,5.)-- (23.,5.);
%%\draw [line width=1pt, draw=ForestGreen] (13.,2.)-- (23.,2.);
%%\draw [line width=1pt, draw=ForestGreen] (8.,-1.)-- (18.,-1.);
%%\draw [line width=1pt, draw=ForestGreen] (3.,-4.)-- (13.,-4.);
%%\draw [line width=1pt, draw=ForestGreen] (-2.,-7.)-- (8.,-7.);
%%\draw[fill=white, draw=black, line width=0.6pt] (23.,5.) ellipse (1.7cm and .9cm);
%%\draw[fill=white, draw=black, line width=0.6pt] (18.,5.) ellipse (1.7cm and .9cm);
\draw[fill=white, draw=black, line width=0.6pt] (13.,2.) ellipse (1.7cm and .9cm);
%%\draw[fill=white, draw=white, line width=0.6pt] (18.,2.) ellipse (1.7cm and .9cm);
\draw[fill=white, draw=black, line width=0.6pt] (23.,2.) ellipse (1.7cm and .9cm);
\draw[fill=white, draw=white, line width=0.6pt] (8.,-1.) ellipse (1.7cm and .9cm);
%%\draw[fill=white, draw=black, line width=0.6pt] (13.,-1.) ellipse (1.7cm and .9cm);
\draw[fill=white, draw=black, line width=0.6pt] (18.,-1.) ellipse (1.7cm and .9cm);
%%\draw[fill=white, draw=white, line width=0.6pt] (8.,-4.) ellipse (1.7cm and .9cm);
\draw[fill=white, draw=white, line width=0.6pt] (13.,-4.) ellipse (1.7cm and .9cm);
\draw[fill=white, draw=black, line width=0.6pt] (8.,-7.) ellipse (1.7cm and .9cm);
%%\draw[fill=white, draw=black, line width=0.6pt] (3.,-7.) ellipse (1.7cm and .9cm);
\draw[fill=white, draw=black, line width=0.6pt] (3.,-4.) ellipse (1.7cm and .9cm);

\fill (-2,-7) circle [radius=0.23];
\fill (3,-10) circle [radius=0.23];
\node at (-2.1,-7,1) {\normalsize $x$};
\node at (2.6,-10.4) {\normalsize $y$};
\node at (3,-4.) {\normalsize $\G_{1,2}$};
\node at (13,2.) {\normalsize $\G_{D-2,D'}$};

%%\node at (3,-7.) {\normalsize $\G_{1,1}$};
\node at (8,-7.) {\normalsize $\G_{2,1}$};
\node at (8,-1.) {\normalsize $\cdots$};
%%\node at (8,-4.) {\normalsize $\cdots$};
%%\node at (23,5.) {\normalsize $\G_{D,D}$};
\node at (23,2.) {\normalsize $\G_{D,D'}$};
%%\node at (18,5.) {\normalsize $\G_{D-1,D'}$};;
%%\node at (18,2.) {\normalsize $\cdots$};
%%\node at (13,-1.) {\normalsize $\G_{i,i}$};
\node at (18,-1.) {\normalsize $\G_{D-1,D'-1}$};
\node at (13,-4.) {\normalsize $\cdots$};
\end{tikzpicture}
}%%%%end of adjustbox
\caption{\rm 
Intersection diagram (of rank $1$) of a $(Y,Y')$-distance-biregular graph with respect to $x\in Y$ and $y\in\G_{1}(x)\subseteq Y'$, where: (a) $\G_{i,j}=\G_{i,j}(x,y)$ $(0\le i,j\le D)$ and $D=D'$ (the left-hand side); (b) $D-1=D'$ (the right-hand side).
}
\label{09}
\end{center}
}\end{figure}}

\subsection{The intersection diagram of rank $\boldsymbol{1}$ of DBG}

For the moment, pick an odd $i$ $(3\le i\le\min\{D,D'\})$, $x\in Y$ and $y\in\G_i(x)$. There are $c_ic_{i-1}\cdots c_1$ different paths of length $i$ between $x$ and $y$. Since $y \in \G_i(x)\subset Y'$ the number of $xy$-paths of length $i$ is also equal to $c'_ic'_{i-1}\cdots c'_1$. Thus
\begin{equation}
\label{Dk}
c_1c_2\cdots c_i=c'_1c'_2\cdots c'_i
\qquad\mbox{for any odd $i$ } (3\le i\le\min\{D,D'\}).
\end{equation}
Similarly, if we count the number of ordered pairs $(x,y)$ such that $x \in Y$ and $y \in \G_i(x)$ in two different ways, we have $|Y|k_i=|Y'|k_i'$. In particular, $|Y|k=|Y'|k'$. Using \eqref{Ej} and \eqref{Dk}, we get
\begin{equation}
\label{Du}
b_1b_2\cdots b_{i-1}=b'_1b'_2\cdots b'_{i-1}
\qquad\mbox{for any odd $i$ } (3\le i\le\min\{D,D'\})
\end{equation}
(see also \cite[Proposition~3]{DC}).

In addition, from the intersection diagram of rank 1 (see Figure~\ref{09}), every vertex from $\G_{1,2}(x,y)$ has exactly $c'_2-1$ neighbours in $\G_{2,1}(x,y)$, and every vertex in $\G_{2,1}(x,y)$ is adjacent to $c_2-1$ vertices in $\G_{1,2}(x,y)$ (for more information about it, see \cite[Lemma 4.3, Lemma 4.4, Lemma 4.5]{FM}). This yields
\begin{equation}
\label{r5}
b_1(c_2-1)=b'_1(c'_2-1). 
\end{equation}
Using the intersection diagram (of rank $1$) from Figure~\ref{09} (see also \cite[Section~4]{FM}), it is routine to compute 
\begin{align}
|\G_{1,0}(x,y)|=1, \qquad& |\G_{i+1,i}(x,y)|=\frac{b_1b_2\cdots b_{i}}{c'_1c'_2\cdots c'_{i}}\ne 0 \qquad(1\le i\le D-1),\label{r2}\\
|\G_{0,1}(x,y)|=1, \qquad& |\G_{i,i+1}(x,y)|=\frac{b'_1b'_2\cdots b'_{i}}{c_1c_2\cdots c_{i}}\ne 0 \qquad(1\le i\le D'-1). \label{r4}
\end{align}
which yields that these two sets are nonempty. All six equations \eqref{Ej}--\eqref{r4} we use later in the paper.

\begin{proposition}
\label{Eo}
With reference to Notation~\ref{GN}, let $\G$ denote $(Y,Y')$-distance-biregular graph with $D\ge 3$. For every integer $i$ $(1\leq i \leq \min\left\lbrace{D-1, D'-1}\right\rbrace)$ the following hold:
\begin{align}
c_{i+1}'=c_i
&\qquad\Lra\qquad
c_{i+1}=c'_i,\label{Es}\\
c_{i+1}'>c_i
&\qquad\Lra\qquad
c_{i+1}>c'_i.\label{Et}
\end{align}
\end{proposition}

\begin{proof}
Pick $x\in Y$, $z\in \G_1(x)$ and an integer $i$ $(1\le i\le\min\left\lbrace{D-1, D'-1}\right\rbrace)$. By \eqref{r2} and \eqref{r4} the sets $\G_{i,i+1}(x,z)$ and $\G_{i+1,i}(x,z)$ are nonempty. Using the intersection diagram (of rank 1) from Figure~\ref{09} it follows that every vertex in $\G_{i,i+1}(x,z)$ has exactly $c_{i+1}'-c_i$ neighbours in $\G_{i+1,i}(x,z)$; and similarly, every vertex in $\G_{i+1,i}(x,z)$ is adjacent to exactly $c_{i+1}-c'_i$ neighbours in $\G_{i, i+1}(x,z)$ (see also \cite[Lemma~4.3, Lemma~4.4, Lemma~4.5]{FM} for further explanation). Thus, $(c_{i+1}'-c_i)|\G_{i, i+1}(x,z)|=|\G_{i+1,i}(x,z)|(c_{i+1}-c'_i)$ and the claim follows.
\end{proof}

\begin{corollary}
\label{Fb}
With reference to Notation~\ref{GN}, let $\G$ denote $(Y,Y')$-distance-biregular graph with $D\ge 3$. If there exists an integer $j$ $(2\leq j \leq \min\left\lbrace{D-1,D'-1}\right\rbrace)$ such that  $c_i=1$ for all $1\leq i\leq j$ then $c'_i=1$ for all $1\leq i\leq j$.
\end{corollary}

\begin{proof}
Using mathematical induction on $i$, the result follows immediately from \eqref{Es}.
\end{proof}

\medskip
For more relations between intersection numbers of the color partitions, see \cite{DC, FM}.

%%%%%%%%%%%%%%%%%%
%%%%%%%%%%%%%%%%%%
%%%%%%%%%%%%%%%%%%
%%%%%%%%%%%%%%%%%%

%%%%%%%%%%%%%%%%%%
%%%%%%%%%%%%%%%%%%
%%%%%%%%%%%%%%%%%%
%%%%%%%%%%%%%%%%%%

\subsection{A $\boldsymbol{(\kappa,g)}$-cage graph}
\label{kH}

{\small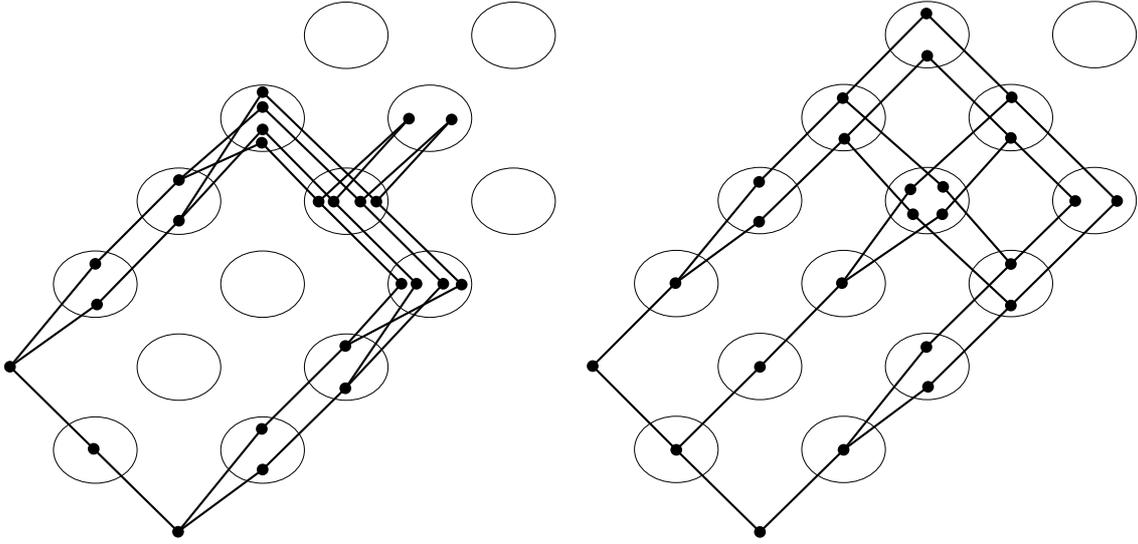
\begin{figure}[t]{\rm
\begin{center}
\begin{tikzpicture}[scale=0.55]
\draw [rotate around={0:(12,2))},line width=.2pt,fill=white,draw=black] (4,8) ellipse (1cm and .8cm);
\draw [rotate around={0:(12,2))},line width=.2pt,fill=white,draw=black] (-4,4) ellipse (1cm and .8cm);
\draw [rotate around={0:(12,2))},line width=.2pt,fill=white,draw=black] (0,0) ellipse (1cm and .8cm);
\draw [rotate around={0:(12,2))},line width=.2pt,fill=white,draw=black] (0,4) ellipse (1cm and .8cm);
\draw [rotate around={0:(12,2))},line width=.2pt,fill=white,draw=black] (-2,6) ellipse (1cm and .8cm);
\draw [rotate around={0:(12,2))},line width=.2pt,fill=white,draw=black] (2,6) ellipse (1cm and .8cm);
\draw [rotate around={0:(12,2))},line width=.2pt,fill=white,draw=black] (0,8) ellipse (1cm and .8cm);
\draw [rotate around={0:(12,2))},line width=.2pt,fill=white,draw=black] (4,4) ellipse (1cm and .8cm);
\draw [rotate around={0:(12,2))},line width=.2pt,fill=white,draw=black] (2,2) ellipse (1cm and .8cm);
\draw [rotate around={0:(12,2))},line width=.2pt,fill=white,draw=black] (-4,0) ellipse (1cm and .8cm);
\draw [rotate around={0:(12,2))},line width=.2pt,fill=white,draw=black] (-6,2) ellipse (1cm and .8cm);
%%\draw [rotate around={0:(12,2))},line width=.2pt,fill=white,draw=black] (-8,0) ellipse (1cm and .8cm);
\draw [rotate around={0:(12,2))},line width=.2pt,fill=white,draw=black] (-6,-2) ellipse (1cm and .8cm);
\draw [rotate around={0:(12,2))},line width=.2pt,fill=white,draw=black] (-2,-2) ellipse (1cm and .8cm);
%%\draw [rotate around={0:(12,2))},line width=.2pt,fill=white,draw=black] (-4,-4) ellipse (1cm and .8cm);
\draw [rotate around={0:(12,2))},line width=.2pt,fill=white,draw=black] (-2,2) ellipse (1cm and .8cm);
\draw [line width=.8pt] (-6,2.49)-- (-8.04,0.01);
\draw [line width=.8pt] (-8.04,0.01)-- (-6.04,-1.97);
\draw [line width=.8pt] (-4,4.51)-- (-6,2.49);
\draw [line width=.8pt] (-5.96,1.51)-- (-8.04,0.01);
\draw [line width=.8pt] (-4.02,-3.97)-- (-6.04,-1.97);
\draw [line width=.8pt] (-2.02,5.41)-- (-4,4.51);
\draw [line width=.8pt] (-4,4.51)-- (-2,6.27);
\draw [line width=.8pt] (-4,3.53)-- (-5.96,1.51);
\draw [line width=.8pt] (-2,-2.47)-- (-4.02,-3.97);
\draw [line width=.8pt] (-4.02,-3.97)-- (-2.02,-1.49);
\draw [line width=.8pt] (-2,5.73)-- (-4,3.53);
\draw [line width=.8pt] (-2,6.63)-- (-4,3.53);
\draw [line width=.8pt] (-2,5.73)-- (-0.3,3.99);
\draw [line width=.8pt] (-2,6.63)-- (0.72,3.99);
\draw [line width=.8pt] (0.72,3.99)-- (2.76,1.99);
\draw [line width=.8pt] (2.52,5.97)-- (0.72,3.99);
\draw [line width=.8pt] (1.68,2.01)-- (-0.3,3.99);
\draw [line width=.8pt] (-0.3,3.99)-- (1.5,5.99);
\draw [line width=.8pt] (1.5,5.99)-- (-0.66,3.99);
\draw [line width=.8pt] (-0.66,3.99)-- (-2.02,5.41);
\draw [line width=.8pt] (-2,6.27)-- (0.34,3.99);
\draw [line width=.8pt] (0.34,3.99)-- (2.52,5.97);
\draw [line width=.8pt] (1.68,2.01)-- (-0.02,-0.51);
\draw [line width=.8pt] (-0.02,-0.51)-- (-2,-2.47);
\draw [line width=.8pt] (-2.02,-1.49)-- (-0.02,0.51);
\draw [line width=.8pt] (-0.02,0.51)-- (2.76,1.99);
\draw [line width=.8pt] (-0.02,-0.51)-- (2.32,2.01);
\draw [line width=.8pt] (2.32,2.01)-- (0.34,3.99);
\draw [line width=.8pt] (-0.02,0.51)-- (1.32,2.01);
\draw [line width=.8pt] (1.32,2.01)-- (-0.66,3.99);
\fill (-6,2.49) circle (4pt);
\fill (-8.04,0.01) circle (4pt);
\fill (-6.04,-1.97) circle (4pt);
\fill (-4,4.51) circle (4pt);
\fill (-5.96,1.51) circle (4pt);
\fill (-4.02,-3.97) circle (4pt);
\fill (-2.02,5.41) circle (4pt);
\fill (-2,6.27) circle (4pt);
\fill (-4,3.53) circle (4pt);
\fill (-2,-2.47) circle (4pt);
\fill (-2.02,-1.49) circle (4pt);
\fill (-0.66,3.99) circle (4pt);
\fill (0.34,3.99) circle (4pt);
\fill (-2,5.73) circle (4pt);
\fill (-2,6.63) circle (4pt);
\fill (1.32,2.01) circle (4pt);
\fill (2.32,2.01) circle (4pt);
\fill (-0.02,-0.51) circle (4pt);
\fill (-0.02,0.51) circle (4pt);
\fill (1.5,5.99) circle (4pt);
\fill (2.52,5.97) circle (4pt);
\fill (-0.3,3.99) circle (4pt);
\fill (0.72,3.99) circle (4pt);
\fill (1.68,2.01) circle (4pt);
\fill (2.76,1.99) circle (4pt);
\end{tikzpicture}~~
\begin{tikzpicture}[scale=0.55]
\draw [rotate around={0:(12,2))},line width=.2pt,fill=white,draw=black] (4,8) ellipse (1cm and .8cm);
\draw [rotate around={0:(12,2))},line width=.2pt,fill=white,draw=black] (-4,4) ellipse (1cm and .8cm);
\draw [rotate around={0:(12,2))},line width=.2pt,fill=white,draw=black] (0,0) ellipse (1cm and .8cm);
\draw [rotate around={0:(12,2))},line width=.2pt,fill=white,draw=black] (0,4) ellipse (1cm and .8cm);
\draw [rotate around={0:(12,2))},line width=.2pt,fill=white,draw=black] (-2,6) ellipse (1cm and .8cm);
\draw [rotate around={0:(12,2))},line width=.2pt,fill=white,draw=black] (2,6) ellipse (1cm and .8cm);
\draw [rotate around={0:(12,2))},line width=.2pt,fill=white,draw=black] (0,8) ellipse (1cm and .8cm);
\draw [rotate around={0:(12,2))},line width=.2pt,fill=white,draw=black] (4,4) ellipse (1cm and .8cm);
\draw [rotate around={0:(12,2))},line width=.2pt,fill=white,draw=black] (2,2) ellipse (1cm and .8cm);
\draw [rotate around={0:(12,2))},line width=.2pt,fill=white,draw=black] (-4,0) ellipse (1cm and .8cm);
\draw [rotate around={0:(12,2))},line width=.2pt,fill=white,draw=black] (-6,2) ellipse (1cm and .8cm);
%%\draw [rotate around={0:(12,2))},line width=.2pt,fill=white,draw=black] (-8,0) ellipse (1cm and .8cm);
\draw [rotate around={0:(12,2))},line width=.2pt,fill=white,draw=black] (-6,-2) ellipse (1cm and .8cm);
\draw [rotate around={0:(12,2))},line width=.2pt,fill=white,draw=black] (-2,-2) ellipse (1cm and .8cm);
%%\draw [rotate around={0:(12,2))},line width=.2pt,fill=white,draw=black] (-4,-4) ellipse (1cm and .8cm);
\draw [rotate around={0:(12,2))},line width=.2pt,fill=white,draw=black] (-2,2) ellipse (1cm and .8cm);
\draw [line width=.8pt] (-8,0.01)-- (-6,-2.01);
\draw [line width=.8pt] (-6,-2.01)-- (-4,-3.99);
\draw [line width=.8pt] (-6.02,2.01)-- (-8,0.01);
\draw [line width=.8pt] (-4,-0.01)-- (-6,-2.01);
\draw [line width=.8pt] (-2,-2.01)-- (-4,-3.99);
\draw [line width=.8pt] (-4.02,4.45)-- (-6.02,2.01);
\draw [line width=.8pt] (-6.02,2.01)-- (-4.02,3.49);
\draw [line width=.8pt] (-2.04,2.01)-- (-4,-0.01);
\draw [line width=.8pt] (0.02,-0.49)-- (-2,-2.01);
\draw [line width=.8pt] (-2,-2.01)-- (-0.02,0.47);
\draw [line width=.8pt] (-0.4,4.27)-- (-2.04,2.01);
\draw [line width=.8pt] (0.36,3.67)-- (-2.04,2.01);
\draw [line width=.8pt] (-0.4,4.27)-- (2.02,6.49);
\draw [line width=.8pt] (0.36,3.67)-- (2,5.51);
\draw [line width=.8pt] (2,5.51)-- (3.54,3.99);
\draw [line width=.8pt] (0,7.49)-- (2,5.51);
\draw [line width=.8pt] (4.54,3.99)-- (2.02,6.49);
\draw [line width=.8pt] (2.02,6.49)-- (-0.02,8.51);
\draw [line width=.8pt] (-0.02,8.51)-- (-2.02,6.47);
\draw [line width=.8pt] (-2.02,6.47)-- (-4.02,4.45);
\draw [line width=.8pt] (-4.02,3.49)-- (-1.98,5.49);
\draw [line width=.8pt] (-1.98,5.49)-- (0,7.49);
\draw [line width=.8pt] (4.54,3.99)-- (2,1.47);
\draw [line width=.8pt] (2,1.47)-- (0.02,-0.49);
\draw [line width=.8pt] (-0.02,0.47)-- (2,2.47);
\draw [line width=.8pt] (2,2.47)-- (3.54,3.99);
\draw [line width=.8pt] (2,1.47)-- (-0.34,3.67);
\draw [line width=.8pt] (-0.34,3.67)-- (-1.98,5.49);
\draw [line width=.8pt] (2,2.47)-- (0.38,4.33);
\draw [line width=.8pt] (0.38,4.33)-- (-2.02,6.47);
\fill (-8,0.01) circle (4pt);
\fill (-6,-2.01) circle (4pt);
\fill (-4,-3.99) circle (4pt);
\fill (-6.02,2.01) circle (4pt);
\fill (-4,-0.01) circle (4pt);
\fill (-2,-2.01) circle (4pt);
\fill (-4.02,4.45) circle (4pt);
\fill (-4.02,3.49) circle (4pt);
\fill (-2.04,2.01) circle (4pt);
\fill (0.02,-0.49) circle (4pt);
\fill (-0.02,0.47) circle (4pt);
\fill (-2.02,6.47) circle (4pt);
\fill (-1.98,5.49) circle (4pt);
\fill (-0.4,4.27) circle (4pt);
\fill (0.36,3.67) circle (4pt);
\fill (0.38,4.33) circle (4pt);
\fill (-0.34,3.67) circle (4pt);
\fill (2,1.47) circle (4pt);
\fill (2,2.47) circle (4pt);
\fill (-0.02,8.51) circle (4pt);
\fill (0,7.49) circle (4pt);
\fill (2.02,6.49) circle (4pt);
\fill (2,5.51) circle (4pt);
\fill (4.54,3.99) circle (4pt);
\fill (3.54,3.99) circle (4pt);
\end{tikzpicture}
\caption{\rm 
Distance-biregular graph with intersection array $(3,1,2,1,2;\,1,1,1,1,2)$ for one color class, and $(2,2,1,2,1,1;\,1,1,1,1,2,2)$ for the other color class.
}
\label{03}
\end{center}
}\end{figure}}

%%A $(\kappa,g)$-cage graph is a regular graph of valency $\kappa$, girth $g$ and minimum number of  vertices $v = (\kappa,g)$. 

If a regular graph of valency $\kappa$ and diameter $d$ has $v$ vertices then $v\le 1+\kappa(\kappa-1)+\cdots+\kappa(\kappa-1)^{d-1}$. Graphs with $v=1+\kappa(\kappa-1)+\cdots+\kappa(\kappa-1)^{d-1}$ are called Moore graphs. {\sc Damerell} \cite{DRM} proved that a Moore graph of valency $\kappa\ge 3$ has diameter $2$. A $(\kappa,g)$-cage graph is a $\kappa$-regular graph of girth $g$ and minimum number of  vertices $v = (\kappa,g)$ (see \cite{NBs,MST}) (constructions of $(\kappa,g)$-cage graph have been studied by many authors, and for more information about it we refer to \cite{EJ,EJS,HK,KZ,NLV,sachs,SHg}). The subdivision graph of a graph $\G$ is the graph obtained from $\G$ by replacing each of its edges by a path of length $2$ (in Figure~\ref{03} the subdivision graph of the Petersen graph is given). The following complete list of $(\kappa,g)$-cage graphs is known (cf. \cite{NBs, MST}): (ci) cycles of length $g$ $(g\ge 3)$; (cii) complete graphs $K_{\kappa+1}$ $(\kappa\ge 3)$; (ciii) complete bipartite graphs $K_{\kappa,\kappa}$ $(\kappa\ge 3)$; (civ) the Petersen graph; (cv) the Hoffman-Singleton graph; (cvi) $(57,5)$-cage graph (existence is disproved recently by {\sc Makhnev} \cite{AAM}); (cvii) a $(\kappa,g)$-cage graph exists with $g\in\{6,8,12\}$, for some values of $\kappa\ge 3$ (not yet completely classified). Note that the subdivision of complete graph $K_{\kappa+1}$ $(\kappa\ge 3)$ has diameter $3$ and the subdivision of complete bipartite graphs $K_{\kappa,\kappa}$ $(\kappa\ge 3)$ has diameter $4$.

In Theorem~\ref{Gp} and Lemma~\ref{ko}, we recall some well-known results of one particular subfamily of distance-biregular graphs which we use later.

\begin{theorem}[{{\rm\cite[Lemma~3.3, Theorem 3.4]{MST}}}]
\label{Gp}
Let $\G =(X ,\R )$ denote a $(\kappa,g)$-cage graph of diameter $d$. The subdivision graph $S(\G )$ of $\G $ is $(X ,\R )$-distance-biregular with the following intersection arrays.
\begin{enumerate}[label=(\alph*), font=\rm]
\item Assume that $g$ is odd. The eccentricity of $x\in X $ in $S(\G )$ is $2d+1$, and the intersection array for $x\in X $ in $S(\G )$ is
$$
(\underbrace{\kappa}_{b_0},\underbrace{1}_{b_1},\kappa-1,1,\kappa-1,\ldots,\underbrace{1}_{b_{2d-1}},\underbrace{\kappa-1}_{b_{2d}};1,1,1,1,\ldots,1,1,\underbrace{2}_{c_{2d+1}}).
$$
The eccentricity of $e\in\R$ in $S(\G)$ is $2d+2$, and the intersection array for $e\in\R $ in $S(\G )$ is
$$
(\underbrace{2}_{b_0'},\underbrace{\kappa-1}_{b_1'},1,\kappa-1,1,\ldots,\kappa-1,\underbrace{1}_{b_{2d}'},\underbrace{\kappa-2}_{b_{2d+1}'}; 1,1,1,1,\ldots,1,1,2,\underbrace{2}_{c'_{2d+2}}).
$$
\item Assume that $g$ is even. The eccentricity of $x\in X $ in $S(\G )$ is $2d$, and the intersection array for $x\in X $ in $S(\G )$ is
$$
(\underbrace{\kappa}_{b_0},\underbrace{1}_{b_1},\kappa-1,1,\kappa-1,\ldots,\underbrace{\kappa-1}_{b_{2d-2}},\underbrace{1}_{b_{2d-1}};
1,1,1,1,\ldots,1,\underbrace{\kappa}_{c_{2d}}).
$$
The eccentricity of $e\in\R $ in $S(\G )$ is $2d$, and the intersection array for $e\in\R $ in $S(\G )$ is
$$
(\underbrace{2}_{b_0'},\underbrace{\kappa-1}_{b_1'},1,\kappa-1,1,\ldots,
\underbrace{1}_{b'_{2d-2}},
\underbrace{\kappa-1}_{b'_{2d-1}}; 1,1,1,1,\ldots,1,1,\underbrace{2}_{c'_{2d}}).
$$
\end{enumerate}
\end{theorem}

\begin{lemma}[{{\rm\cite[Corollary 3.5]{MST}}}]
\label{ko}
With reference to Notation~\ref{GN}, a graph $\G$ with vertices of valency $2$ is distance-biregular if and only if $\G$ is either a complete bipartite graph $K_{2,n}$ (for $n\geq 1$) or the subdivision graph of a $(\kappa,g)$-cage graph.
\end{lemma}

%%%%%%%%%%%%%%%%%%%%%%%%%%%%%%%%%%%%%%%%%%%
%%%%%%%%%%%%%%%%%%%%%%%%%%%%%%%%%%%%%%%%%%%
%%%%%%%%%%%%%%%%%%%%%%%%%%%%%%%%%%%%%%%%%%%
%%%%%%%%%%%%%%%%%%%%%%%%%%%%%%%%%%%%%%%%%%%
%%%%%%%%%%%%%%%%%%%%%%%%%%%%%%%%%%%%%%%%%%%
%%%%%%%%%%%%%%%%%%%%%%%%%%%%%%%%%%%%%%%%%%%

\subsection{The intersection diagram of rank $\boldsymbol{2}$ and the scalars $\boldsymbol{p^i_{2,i}}$ and $\boldsymbol{p^2_{i,i}}$}

For a $(Y,Y')$-distance-biregular graph $\G$ with $D\ge 2$, pick $x\in Y$, $y\in\G_2(x)$ and consider the intersection diagram from Figure~\ref{02}. Using the structure of $\G$ from this intersection diagram, and the fact that $\G$ is distance-biregular, it is not hard to see that for every $i$ $(1\le i\le D)$
\begin{align}
\G_i(x)&=\G_{i,i+2}(x,y) \cup \G_{i,i}(x,y) \cup \G_{i,i-2}(x,y),\label{id0}\\
z\in \G_{i,i-2}(x,y)\cup\G_{i-2,i}(x,y)\qquad
&\Rightarrow\qquad 
\left| \G_1(z) \cap \G_{i-1,i-1}(x,y)\right| = c_i - c_{i-2} = b_{i-2} - b_i, \label{id6}\\
\G_{1,1}(x,y)&\ne\emptyset
\qquad\mbox{and}\qquad
\G_{D-1,D-1}(x,y)\ne\emptyset.\label{Dj}
\end{align}
For example, if $\G_{D-1,D-1}(x,y)=\emptyset$ then vertices in $\G_{D-2,D}(x,y)$  have no neighbours in $\G_{D-1}(x)$, that is $b_{D-2}=0$, a contradiction with the definition of a distance-biregular graph. From the same diagram, it is routine to compute
\begin{align}
|\G_{2,0}(x,y)|=1,\qquad
|\G_{i,i-2}(x,y)|&=\frac{b_2b_3\cdots b_{i-1}}{c_1c_2\cdots c_{i-2}}\ne 0 \qquad(3\le i\le D),\label{am}\\
|\G_{0,2}(x,y)|=1,\qquad
|\G_{i,i+2}(x,y)|&=\frac{b_2b_3\cdots b_{i+1}}{c_1c_2\cdots c_{i}}\ne 0 \qquad(1\le i\le D-2),\label{ap}\\
|\G_{2,2}(x,y)|&=\frac{1}{c_2}(b_0b_1-b_2b_3-c_2)=
\frac{1}{c_2}(c_2(b_1-1)+b_2(c_3-1))\label{Ea}\\
|\G_{i,i}(x,y)|&=|\G_i(x)|-|\G_{i,i-2}(x,y)|-|\G_{i,i+2}(x,y)|\nonumber\\
~&=\frac{b_2b_3\cdots b_{i-1}}{c_1c_2\cdots c_i}(kb_1-b_ib_{i+1}-c_{i-1}c_i)\qquad(3\le i\le D-2),\label{aq}\\
|\G_{D-1,D-1}(x,y)|&=|\G_{D-1}(x)|-|\G_{D-1,D-3}(x,y)|\nonumber\\
~&=\frac{b_2b_3\cdots b_{D-2}}{c_1c_2\cdots c_{D-1}}(kb_1-c_{D-2}c_{D-1}),\label{au}\\
|\G_{D,D}(x,y)|&=|\G_{D}(x)|-|\G_{D,D-2}(x,y)|\nonumber\\
~&=\frac{b_2b_3\cdots b_{D-1}}{c_1c_2\cdots c_{D}}(kb_1-c_{D-1}c_{D}).\label{av}
\end{align}

\begin{figure}[t]
\begin{center}
{\rm{\small
\begin{tikzpicture}[scale=0.55]
\draw [line width=1pt, draw=ForestGreen] (-12,-2)-- (-4,-6);
\draw [line width=1pt, draw=ForestGreen] (-4,-6)-- (16,4);
\draw [line width=1pt, draw=ForestGreen] (16,4)-- (8,8);
\draw [line width=1pt, draw=ForestGreen] (8,8)-- (-12,-2);
\draw [line width=1pt, draw=ForestGreen] (4,6)-- (12,2);
\draw [line width=1pt, draw=ForestGreen] (0,4)-- (8,0);
\draw [line width=1pt, draw=ForestGreen] (-4,2)-- (4,-2);
\draw [line width=1pt, draw=ForestGreen] (-8,0)-- (0,-4);
\draw [line width=1pt, draw=ForestGreen] (-8,-4)-- (16,8);

\fill (-12,-2) circle [radius=0.17];
\node at (-12.5,-2.5) {$x$};
\fill (-4,-6) circle [radius=0.17];
\node at (-4.5,-6.5) {$y$};

\draw [rotate around={0:(-8,-4)},line width=.6pt,fill=white,draw=black] (-8,-4) ellipse (1.7cm and .9cm);
\node at (-8,-4) {$\G_{1,1}$};
%%\fill (-8,-4) circle [radius=0.14];
\draw [rotate around={0:(-8,0)},line width=.6pt,fill=white,draw=black] (-8,0) ellipse (1.7cm and .9cm);
\node at (-8,0) {$\G_{1,3}$};
%%\fill (-8,0) circle [radius=0.14];
\draw [rotate around={0:(-4,-2))},line width=.6pt,fill=white,draw=black] (-4,-2) ellipse (1.7cm and .9cm);
\node at (-4,-2) {$\G_{2,2}$};
%%\fill (-4,-2) circle [radius=0.14];
\draw [rotate around={0:(-4,2))},line width=.6pt,fill=white,draw=black] (-4,2) ellipse (1.7cm and .9cm);
\node at (-4,2) {$\G_{2,4}$};
%%\fill (-4,2) circle [radius=0.14];

\draw [rotate around={0:(0,0))},line width=.6pt,fill=white,draw=white] (0,0) ellipse (1.7cm and .9cm);
\node at (0,0) {$\boldsymbol\ldots$};
%%\fill (0,0) circle [radius=0.14];
\draw [rotate around={0:(0,-4))},line width=.6pt,fill=white,draw=white] (0,-4) ellipse (1.7cm and .9cm);
\node at (0,-4) {$\boldsymbol\ldots$};
%%\fill (0,-4) circle [radius=0.14];
\draw [rotate around={0:(0,4))},line width=.6pt,fill=white,draw=white] (0,4) ellipse (1.7cm and .9cm);
\node at (0,4) {$\boldsymbol\ldots$};
%%\fill (0,4) circle [radius=0.14];

\draw [rotate around={0:(4,-2))},line width=.6pt,fill=white,draw=black] (4,-2) ellipse (1.7cm and .9cm);
\node at (4,-2) {$\G_{i,i-2}$};
%%\fill (4,-2) circle [radius=0.14];
\draw [rotate around={0:(4,2))},line width=.6pt,fill=white,draw=black] (4,2) ellipse (1.7cm and .9cm);
\node at (4,2) {$\G_{i,i}$};
%%\fill (4,2) circle [radius=0.14];
\draw [rotate around={0:(4,6))},line width=.6pt,fill=white,draw=black] (4,6) ellipse (1.7cm and .9cm);
\node at (4,6) {$\G_{i,i+2}$};
%%\fill (4,6) circle [radius=0.14];

\draw [rotate around={0:(8,0))},line width=.6pt,fill=white,draw=white] (8,0) ellipse (1.7cm and .9cm);
\node at (8,0) {$\boldsymbol\ldots$};
%%\fill (8,0) circle [radius=0.14];
\draw [rotate around={0:(8,4))},line width=.6pt,fill=white,draw=white] (8,4) ellipse (1.7cm and .9cm);
\node at (8,4) {$\boldsymbol\ldots$};
%%\fill (8,4) circle [radius=0.14];
\draw [rotate around={0:(8,8))},line width=.6pt,fill=white,draw=white] (8,8) ellipse (1.7cm and .9cm);
\node at (8,8) {$\boldsymbol\ldots$};
%%\fill (8,8) circle [radius=0.14];

\draw [rotate around={0:(12,2))},line width=.6pt,fill=white,draw=black] (12,2) ellipse (1.7cm and .9cm);
\node at (12,2) {$\G_{D-1,D-3}$};
%%\fill (12,2) circle [radius=0.14];
\draw [rotate around={0:(12,6))},line width=.6pt,fill=white,draw=black] (12,6) ellipse (1.7cm and .9cm);
\node at (12,6) {$\G_{D-1,D-1}$};
%%\fill (12,6) circle [radius=0.14];
\draw [rotate around={0:(16,4))},line width=.6pt,fill=white,draw=black] (16,4) ellipse (1.7cm and .9cm);
\node at (16,4) {$\G_{D,D-2}$};
%%\fill (16,4) circle [radius=0.14];
\draw [rotate around={0:(16,8))},line width=.6pt,fill=white,draw=black] (16,8) ellipse (1.7cm and .9cm);
\node at (16,8) {$\G_{D,D}$};
%%\fill (16,8) circle [radius=0.14];
\end{tikzpicture}
}}
\caption{\rm 
Intersection diagram of rank $2$ of a distance-biregular graph with respect to $x\in X$ and $y\in\G_{2}(x)$, where $\G_{i,j}=\G_{i,j}(x,y)$ $(0\le i,j\le D)$.
}
\label{02}
\end{center}
\end{figure}
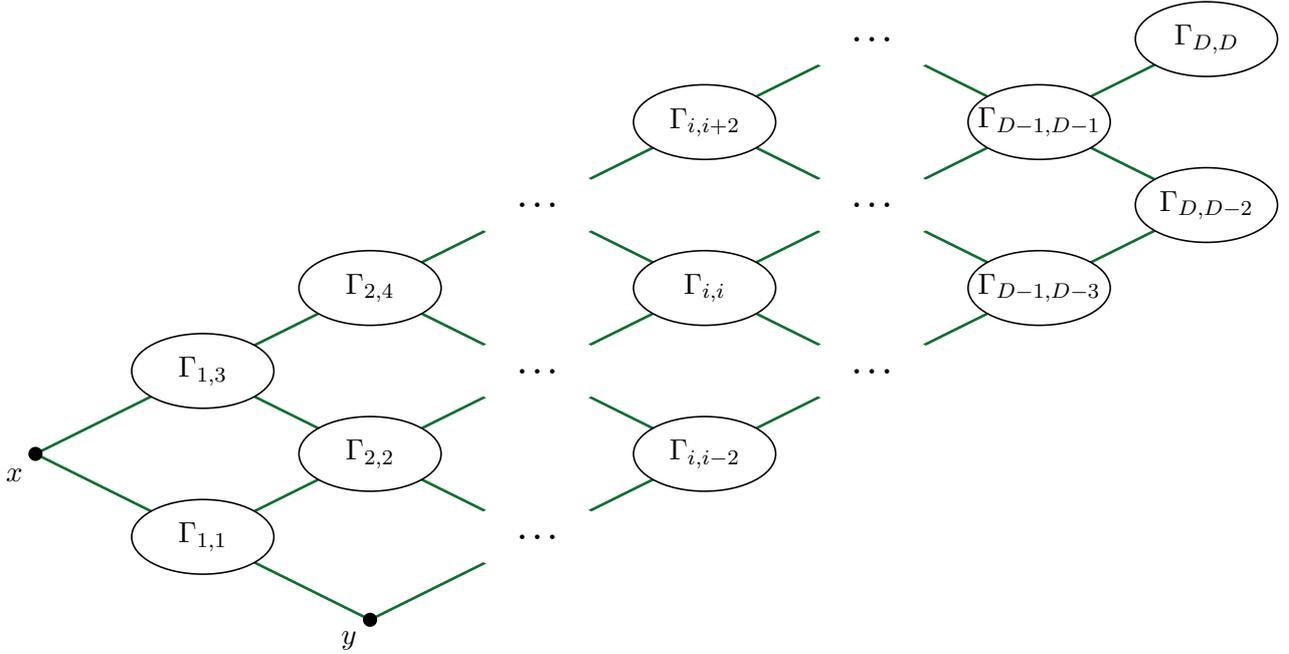

\noindent
Now fix $i$ $(1 \le i \le D)$, and pick $z\in\G_{i}(x)$. We define the numbers $p^i_{2,i}$ and $p^2_{i,i}$ as follows:
$$
p^i_{2,i}=p^i_{2,i}(x,z):=|\G_2(x)\cap\G_{i}(z)|,
\qquad
p^2_{i,i}=p^2_{i,i}(x,y):=|\G_i(x)\cap\G_{i}(y)|.
$$
Fixing $x\in Y$, and counting the number of ordered pairs $(y,z)\in X\times X$, where $y\in\G_{2}(x)$, $z\in\G_{i}(x)$ and $\partial(y,z)=i$, it is routine to show that
\begin{equation}
\label{Di}
k_2p^2_{i,i}=k_ip^i_{2,i}.
\end{equation}
Note that $p^2_{i,i}=0$ if and only if $p^i_{2,i}=0$, $p^1_{2,1}=b_1$ and $p^2_{1,1}=c_2$. From \eqref{Ej}, \eqref{Ea}--\eqref{Di}, we have that the numbers $p^i_{2,i}(x,z)$ and $p^2_{i,i}(x,y)$ do not depend on the choices of $x\in Y$, $y\in\G_2(x)$ and $z\in\G_{i}(x)$, but only on the value of $i$.

\begin{lemma}
\label{Dc}
With reference to Notation~\ref{GN}, pick $i$ $(2\le i\le D-2)$. Then
\begin{align}
p^2_{i,i}=0 
\qquad&
%%\Leftrightarrow
\mbox{if and only if}
\qquad
kb_1-b_ib_{i+1}-c_{i-1}c_i=0.\label{Dl}
\end{align}
Furthermore, $p^2_{D-1,D-1}\ne 0$, and $p^2_{D,D}= 0$ if and only if either $D$ is even and $b_{D-1}=1$, or $D$ is odd and $b'_{D-1}=1$ (in the case when $D$ is odd, we also have $D=D'$).
\end{lemma}

\begin{proof}
Note that $p^2_{i,i}=|\G_{i,i}(x,y)|$ $(2\le i\le D)$. Equation \eqref{Dl} follows immediately from \eqref{aq}. The claim $p^2_{D-1,D-1}\ne 0$ follows from \eqref{Dj}. The proof for the third claim we split in two cases.

{\sc Case 1}. Assume that $D$ is even. By \eqref{av}, $p^2_{D,D}=0$ if and only if $kb_1=c_{D-1}c_D$. Since $c_D=k$ we have $p^2_{D,D}=0$ if and only if $b_1=c_{D-1}$. On the other hand $D-1$ is odd, and we have $c_{D-1}+b_{D-1}=k'=b_1+1=c_{D-1}+1$. The result follows.

{\sc Case 2}. Assume that $D$ is odd. Pick $x\in Y$, and note that $\G_D(x)\subset Y'$. This yields $D'\ge D$. Since $\min\{D,D'\}=D$, by \eqref{Dk} we have $c_{D-1}c_D=c'_{D-1}c'_D$. Now \eqref{av} yields $p^2_{D,D}=0$ if and only if $kb_1=c'_{D-1}c'_D$. Note that $b'_{D-1}+c'_{D-1}=k'=b_1+1$, and with it $b'_{D-1}-1=b_1-c'_{D-1}$. As $c'_D+b'_D=k$ we have $p^2_{D,D}=0$ if and only if $c'_D(b_1-c'_{D-1})+b_D'b_1=0$. 
If $D'=D+1$ then $c'_D(b_1-c'_{D-1})+b_D'b_1= c'_D(b'_{D-1}-1)+b_D'b_1=0$, a contradiction (the numbers $b_1$, $b'_{D-1}$ and $b'_D$ are positive). Thus we have $D'=D$ and $b'_D=0$. With it $p^2_{D,D}=0$ if and only if $b_1=c'_{D-1}$, and the result follows.
\end{proof}

%%%%%%%%%%%%%%%%%%%%%%%%%%%%%%%%%
%%%%%%%%%%%%%%%%%%%%%%%%%%%%%%%%%
%%%%%%%%%%%%%%%%%%%%%%%%%%%%%%%%%
%%%%%%%%%%%%%%%%%%%%%%%%%%%%%%%%%
%%%%%%%%%%%%%%%%%%%%%%%%%%%%%%%%%
%%%%%%%%%%%%%%%%%%%%%%%%%%%%%%%%%
%%%%%%%%%%%%%%%%%%%%%%%%%%%%%%%%%
%%%%%%%%%%%%%%%%%%%%%%%%%%%%%%%%%
%%%%%%%%%%%%%%%%%%%%%%%%%%%%%%%%%
%%%%%%%%%%%%%%%%%%%%%%%%%%%%%%%%%
%%%%%%%%%%%%%%%%%%%%%%%%%%%%%%%%%
%%%%%%%%%%%%%%%%%%%%%%%%%%%%%%%%%

\section{The scalars $\boldsymbol{\gamma_i}$ and the case when $\boldsymbol{c_2'\ge 2}$}
\label{La}

Let $\G$ denote an almost $2$-$Y$-homogeneous $(Y,Y')$-distance-biregular graph with $D\ge 3$. In this section we define scalars $\gamma_i$ in the following way. Pick $x\in Y$, $y\in\G_2(x)$ and $u\in\G_{i,i}(x,y)$. Then 
$$
\gamma_i:=|\G_{i-1}(u)\cap\G_{1,1}(x,y)|\qquad (1\leq i \leq D-2).
$$
In addition, if $\G$ is a $2$-$Y$-homogeneous then we also define $\gamma_{D-1}$ by
$$
\gamma_{D-1}:=|\G_{D-2}(u)\cap\G_{1,1}(x,y)|.
$$
Since we are working with an (almost) $2$-$Y$-homogeneous DBG, we are not interested in defining $\gamma_D$. Clearly we have $\gamma_1=1$. In the case when $D-1=D'$ we have $\gamma_{D-1}=c_2$ (see the proof of Corollary~\ref{Ek}). In Section~\ref{oH} we show that the scalars $\gamma_i$ are nonzero, when $k'\ge 3$.

In this section we compute some equalities for the case $c_2'\ge 2$, and we show that, if $c_2'\ge 3$ then $D\le 5$.

\begin{lemma}
\label{gH}
With reference to Notation~\ref{GN}, let $\G$ denote a $2$-$Y$-homogeneous $(Y,Y')$-distance-biregular graph with $D\ge 3$. Pick $i$ $(1\leq i \leq \min\left\lbrace D-1, D'-1 \right\rbrace)$, and assume that $\gamma_2$ and $\gamma_i$ are nonzero.
\begin{enumerate}[label=(\roman*), font=\rm]
\item If $c_2'\ge 2$ then $(k'-2)(\gamma_2-1)=(c_2-1)(c_2'-2)$.
\item If $i$ is even then $\gamma_i(c_{i+1}-1)=c_i(c_2'-1)$.
\item If $i$ is odd then $\gamma_i(c_{i+1}'-1)=c_i'(c_2'-1)$.
\end{enumerate}
\end{lemma}

\begin{proof}
Our proof is along the lines of the proof of \cite[Lemma~5.1]{NKsm}, where the author studied bipartite distance-regular graphs.

(i) Note that, by assumption we have $k'\ge 3$ (since $k'=c_2'+b_2'$ and $D'\ge 3$). Pick $u\in Y$, $v\in\G_1(u)$ and $w\in\G_{2,1}(u,v)$. We count the number $N$ of ordered pairs $(x,y)$ with $x\in\G_{1,2,1}(u,v,w)$ and $y\in\G_{2,1,2}(u,v,w)$ in two different ways. Since $w\in\G_2(u)$, there are precisely $c_2-1$ vertices $x\in\G_{1,1}(u,w)$ with $x\ne v$. Fix such a vertex $x$. Since $x\in\G_2(v)\subseteq Y'$, there are precisely $c_2'-2$ vertices $y\in\G_{1,1}(x,v)$ with $y\ne u$, $y\ne w$. So we have $N=(c_2-1)(c_2'-2)$. On the other hand, there are precisely $k'-2$ vertices $y\in\G_{2,1}(u,v)$ with $y\ne w$. Fix such a vertex $y$ and note that $y\in\G_2(w)$. Since $\partial(y,w)=2$ and $u\in\G_{2,2}(y,w)$, $w$ and $y$ have precisely $\gamma_2-1$ common neighbours $x\in\G_{1}(u)$ with $x\ne v$. So we obtain $N=(k'-2)(\gamma_2-1)$.

(ii) Assume that $i$ is even. Pick $u\in Y$, $v\in\G_i(u)$ and $w\in\G_{i+1,1}(u,v)$. The result follows by counting the number of ordered pairs $(x,y)$ with $x\in\G_{i-1,1}(u,v)$ and $y\in\G_{i,2,1}(u,v,w)$ in two different ways.

(iii) Assume that $i$ is odd. Pick $u\in Y'$, $v\in\G_i(u)$ and $w\in\G_{i+1,1}(u,v)$. The result follows by counting the number of ordered pairs $(x,y)$ with $x\in\G_{i-1,1}(u,v)$ and $y\in\G_{i,2,1}(u,v,w)$ in two different ways.
\end{proof}

\begin{remark}
\label{iH}
{\rm
Note that all three claims of Lemma~\ref{gH} hold also for an almost $2$-$Y$-homogeneous distance-biregular graph, under the assumption that $D\ge 4$ and $1\le i\le D-2$.
}\end{remark}

\begin{theorem}
\label{hH}
With reference to Notation~\ref{GN}, let $\G$ denote a $2$-$Y$-homogeneous $(Y,Y')$-distance-biregular graph with $D\ge 3$. If $c_2'\ge 3$ then $D\le 5$.
\end{theorem}

\begin{proof}
If $c_2'\ge 3$ then \eqref{r5} yields $c_2\ge 2$. Note that $c_2'+b_2'=k'\ge 3$ and by Lemma~\ref{gH}(i), $\gamma_2\ge 2$.

Our proof is by contradiction. Assume that $D\ge 6$. Lemma~\ref{jH} yields $c_3\le b_3$; that is, $2c_3\le k'$, or expressed differently, $2(c_3-1)\le k'-2$. By Lemma~\ref{gH}(i)(ii) this yields
$$
2\cdot \frac{c_2(c_2'-1)}{\gamma_2}\le \frac{(c_2-1)(c_2'-2)}{\gamma_2-1}.
$$
Multiplying both sides by $\gamma_2(\gamma_2-1)$ we get
$$
2c_2(c_2'-1)(\gamma_2-1)\le (c_2-1)(c_2'-2)\gamma_2.
$$
Since $(c_2-1)(c_2'-2)\gamma_2 < c_2(c_2'-1)\gamma_2$ we have $2(\gamma_2-1)<\gamma_2$, that is, $\gamma_2<2$ a contradiction.
\end{proof}

\begin{remark}
\label{lH}
{\rm
Let $\G$ denote a $2$-$Y$-homogeneous $(Y,Y')$-distance-biregular graph with $D\ge 3$. From the proof of Theorem~\ref{hH} we also have that $c_2'\ge 3$ yields $b_3<c_3$ (assumption $c_3\le b_3$ yields $\gamma_2<2$, a contradiction).
}\end{remark}

\begin{lemma}
\label{vH}
With reference to Notation~\ref{GN}, let $\G$ denote a $2$-$Y$-homogeneous $(Y,Y')$-distance-biregular graph with $D\ge 3$ and $k'\ge 2$. If $c_2'\ge 2$ then the following hold.
\begin{enumerate}[label=(\roman*), font=\rm]
\item $\gamma_2\ge 1$.
\item $c_3=\frac{c_2(c_2'-1)}{\gamma_2}+1$.
\item $(k'-2)(c_2-\gamma_2)=(k-c_2)(c_2'-1)$.
\end{enumerate}
\end{lemma}

\begin{proof}
(i) Since $c_2'\ge 2$, by Proposition~\ref{Eo} we have $c_2\ge 2$. Lemma~\ref{4w}(i) implies $c_3\ge 2$. Pick $u\in Y$, $v\in\G_2(u)$ and $w\in\G_{3,1}(u,v)$. By counting the number $N$ of ordered pairs $(x,y)$ with $x\in\G_{i-1,1}(u,v)$ and $y\in\G_{i,2,1}(u,v,w)$ in two different ways, we get $c_2(c_2'-1)=N=\gamma_2(c_{3}-1)$ (see also the proof of Lemma~\ref{gH}(ii)). This yields $\gamma_2\ge 1$.

(ii) By the comment from (i) we have $c_3-1=\frac{c_2(c_2'-1)}{\gamma_2}$, and the result follows.

(iii) Pick $x\in Y$ and $y\in\G_2(x)$. By \eqref{Ea} and (ii) above,
\begin{equation}
\label{tH}
|\G_{2,2}(x,y)|=p^2_{22}=k'-2+\frac{b_2(c_2'-1)}{\gamma_2}.
\end{equation}
On the other hand, since every vertex from $\G_{2,2}(x,y)$ has exactly $\gamma_2$ neighbours in $\G_{1,1}(x,y)$, we have
\begin{equation}
\label{uH}
|\G_{2,2}(x,y)|=\frac{c_2(k'-2)}{\gamma_2}.
\end{equation}
Using \eqref{tH} and \eqref{uH}, the result follows.
\end{proof}

\section{Distance-biregular graph with $\boldsymbol{k'=2}$}
\label{D2}

By \cite[Corollary~3.5]{MST}, a graph $\G$ with $2$-valent vertices is distance-biregular if and only if either $\G=K_{2,r}$, or $\G$ is the subdivision graph of a $(\kappa,g)$-graph. In this section we show that a $(Y,Y')$-distance-biregular graph with $k'=2$ is $2$-$Y$-homogeneous, and we give combinatorial properties of such graphs. As a corollary, we get that the subdivision graph of a $(\kappa,g)$-cage graph $\G $ (with vertex set $X $) is $2$-$X $-homogeneous. The combinatorial structure of distance-biregular graphs with $k'=2$ plays an important role later in the paper.

\begin{proposition}
\label{km}
With reference to Notation~\ref{GN}, let $\G$ denote a $(Y,Y')$-distance-biregular graph. If $D=3$ then the following are equivalent.
\begin{enumerate}[label=(\roman*), font=\rm]
\item $\G$ is the subdivision graph of a complete graph $K_n$ ($n\ge 3$). 
\item $k'=2$.
\end{enumerate}
Moreover, if {\rm (i), (ii)} hold with $Y$ as the set of vertices of the complete graph $K_n$, then $\G$ is $2$-$Y$-homogeneous, and the intersection array of the color class $Y$ is $(k,1,k-1;1,1,2)$.
\end{proposition}

\begin{proof}
(i)$\Ra$(ii) Immediate from the definition of subdivision.

\medskip
(ii)$\Ra$(i) Assume that $\G$ is a distance-biregular graph with vertices in $Y'$ of valency $2$. Since $D=3$, by Lemma~\ref{ko}, $\G$ is the subdivision graph of a $(\kappa, g)$-graph $\G^\circ=(X^\circ,\R^\circ)$. By Theorem~\ref{Gp}, the girth $g$ of $\G^\circ$ is odd as otherwise vertices of $\G^\circ$ have even eccentricity. Even more, all vertices of eccentricity $3$ are those which lie in $X^\circ$. This shows that diameter $d$ of $\G^\circ$ is equal to $1$ (as every vertex in $X^\circ$ has eccentricity $2d+1$). Thus, $\G^\circ$ is a complete graph $K_n$ $(n\geq 3)$, and the result follows.

\smallskip
Assume now that {\rm (i), (ii)} hold. By Theorem~\ref{Gp}, vertices in $\G$ of the color partition $Y$ have intersection array $(n-1,1,n-2;1,1,2)$. Pick $x\in Y$, $y\in\G_2(x)$ and consider the combinatorial structure of $\G$ from Figure~\ref{02}. Since $c_2=1$ and $k'=2$, the unique vertex of $\G_{1,1}(x,y)$ does not have neighbours in $\G_{2,2}(x,y)$. Therefore, for every $x\in Y$, $y\in\G_2(x)$ and $z\in \G_{2,2}(x,y)$, $\left|\G_1(z)\cap \G_{1,1}(x,y)\right|=0$. The result follows.
\end{proof}

\begin{theorem}
\label{Db}
With reference to Notation~\ref{GN}, let $\G$ denote a $(Y,Y')$-distance-biregular graph. If $D\ge 4$ then the following are equivalent.
\begin{enumerate}[label=(\roman*), font=\rm]
\item Either $\G$ is the subdivision graph of a $(\kappa,g)$-cage graph $\G^\circ=(X^\circ,\R^\circ)$ (where $\G^\circ$ has diameter $d\ge 2$, valency $\kappa\geq 3$ and girth $g\geq 3$) with $Y=X^\circ$ or it is an even-length cycle. 
\item For all $x\in Y$ and $y\in \G_2(x)$ the sets $\G_{i,i}(x,y)$ $(2\le i \leq D-2)$ are empty.
\item There exists $i$ $(2\le i \leq D-2)$ such that $p^2_{i,i}=0$.
\item $k'=2$.
\end{enumerate}
Moreover, if {\rm (i)--(iv)} hold with $Y$ as the set of vertices of a $(\kappa,g)$-cage graph, then $\G$ is $2$-$Y$-homogeneous, and the intersection array of the color class $Y$ is one of the follow two types: 
\begin{align*}
(k,1,k-1,1,k-1,\ldots,1,k-1;1,1,1,1,\ldots,1,1,2),\qquad \mbox{ for odd } D\\
(k,1,k-1,1,k-1,\ldots,k-1,1;1,1,1,1,\ldots,1,k),\qquad \mbox{ for even } D.
\end{align*}
\end{theorem}

\begin{proof} 
(i)$\Rightarrow$(ii) If $\G^\circ$ is an even-length cycle the result immediately  follows. Let $\G^\circ=(X^\circ,\mathcal{R}^\circ)$ denote a $(\kappa,g)$-graph with diameter $d\ge 2$, $\kappa\geq 3$ and $g\geq 3$. Assume that $\G$ is the subdivision graph of $\G^\circ$ and recall that by assumption $D\ge 4$. The color partitions of $\G$ are $Y=X^\circ$ and  $Y'=\mathcal{R}^\circ$. Moreover, the intersection arrays of $\G$ are given in Theorem~\ref{Gp} and they depend on the parity of $g$. Thus, two cases must be considered.

{\sc Case 1.} Assume that $g$ is odd. Pick $x\in X^\circ=Y$, and note that $D=2d+1$. For every $i$ $(2\leq i \leq 2d-1)$ we have $b_ib_{i+1}=\kappa-1$ and $c_{i-1}c_i=1$. Since $b_1=1$, it follows that $\kappa b_1-b_ib_{i+1}-c_{i-1}c_i=0$ $(2\le i\le 2d-1=D-2)$. The result follows from Lemma~\ref{Dc}.

{\sc Case 2.} Assume that $g$ is even. Pick $x\in X^\circ=Y$, and note that $D=2d$. For every $i$ $(2\leq i \leq 2d-2)$ we have $b_ib_{i+1}=\kappa-1$ and $c_{i-1}c_i=1$. Since $b_1=1$, it follows that $\kappa b_1-b_ib_{i+1}-c_{i-1}c_i=0$ $(2\le i\le 2d-2=D-2)$. The result follows from Lemma~\ref{Dc}.

\medskip
(ii)$\Rightarrow$(iii) Trivial.

\medskip
(iii)$\Rightarrow$(iv) Let $x\in Y$, $y\in \G_2(x)$, pick $i$ $(2\le i\le D-2)$ and assume that $\G_{i,i}(x,y)$ is empty. We split the proof in two cases.

{\sc Case 1.} Assume that $i$ is even. By \eqref{Ea} and \eqref{aq},
$\G_{i,i}(x,y)=\emptyset$ if and only if $kb_1-b_ib_{i+1}-c_{i-1}c_i=0$. Since $b_i+c_i=k$, we have 
\begin{equation}
\label{De}
kb_1-b_ib_{i+1}-c_{i-1}c_i=c_{i}(b_1-c_{i-1})+b_{i}(b_1-b_{i+1})=0. 
\end{equation}
Note that $b_1=k'-c_1=k'-1\ge k'-b_{i-1}=c_{i-1}$. Also, by Lemma~\ref{4w}, we have $b_{i+1}\leq b_i'\leq b_{i-1}$ which yields $b_1\ge b_{i+1}$. Equation \eqref{De} now yields $b_{i+1}=b_1=c_{i-1}$. On the other hand, since $b_1=k'-1$, we have 
\begin{equation}
\label{Df}
c_{i-1}=b_{i+1}=k'-1.
\end{equation}
Note that, as $i$ is even, the integers $i-1$ and $i+1$ are odd, and we have $c_{i-1}+b_{i-1}=k'$ and $c_{i+1}+b_{i+1}=k'$. Now by \eqref{Df}, we get $b_{i-1}=c_{i+1}=1$, and since $c_{i-1}\le c_i'\le c_{i+1}$ and $b_{i+1}\le b_i'\le b_{i-1}$ (see Lemma~\ref{4w}) we have
\begin{equation}
\label{Dg}
c_{i-1}=b_{i+1}=1.
\end{equation}
The result now follows from \eqref{Df} and \eqref{Dg}.

{\sc Case 2.} Assume that $i$ is odd. By \eqref{id6}, note that every vertex $z\in\G_{i+1,i-1}(x,y)$ has exactly $c_{i+1}-c_{i-1}$ neighbours in $\G_{i,i}(x,y)$. Since by our assumption $\G_{i,i}(x,y)$ is empty, it must be $c_{i+1}=c_{i-1}$. Now, $k=b_{i+1}+c_{i+1}=b_{i+1}+c_{i-1}$ (since $i+1$ is even) and we have
\begin{equation}
\label{Dh}
kb_1-b_ib_{i+1}-c_{i-1}c_i=c_{i-1}(b_1-c_i)+b_{i+1}(b_1-b_i)=0.
\end{equation}
Since $b_1+c_1=k'=b_i+c_i$, $b_1=k'-1\ge k'-b_i=c_i$. Now, similarly as in Case~1 (from Lemma~\ref{4w}) it is not hard to see that $b_1\ge b_i$. Using the last two facts together with \eqref{Dh} we have $b_1=c_i=b_i$. In the end, since $k'-1=b_1=b_i=k'-c_i$ and $k'-1=b_1=c_i=k'-b_i$ we get $c_i=1=b_i$, and the result follows.

\medskip
(iv)$\Rightarrow$(i) By assumption $D\ge 4$, so $\G$ is not a complete bipartite graph $K_{2,n}$ for some $n\ge 1$. By Lemma~\ref{ko} it follows that $\G$ is either the subdivision graph of $(\kappa,g)$-cage graph (with diameter $d\ge 2$, valency $\kappa\geq 3$ and girth $g\geq 3$) or it is an even-length cycle. The result follows.
\end{proof}

\begin{corollary}
\label{Eu}
Let $\G^\circ$ denote a $(\kappa,g)$-cage graph for integers $\kappa \geq 2$ and $g\geq 3$, with vertex set $X^\circ$. Then, the subdivision graph of $\G^\circ$ is $2$-$X^\circ$-homogeneous.
\end{corollary}

\begin{proof}
Let $\G^\circ=(X^\circ,\R^\circ)$ denote a $(\kappa,g)$-cage graph with $\kappa \geq 2$ and $g\geq 3$. By Theorem~\ref{Gp}, the subdivision graph $\G=S(\G^\circ)$ (of $(\kappa,g)$-cage graph $\G^\circ$) is distance-biregular with $k'=2$ and color partitions $X^\circ$ and $\R^\circ$. Pick $x\in X^\circ$ from $\G=S(\G^\circ)$, and consider two cases: {\sc Case 1.} Assume that $k=2$. Then $\G$ is an even length cycle, and therefore $\G$ is $2$-$X^\circ$-homogeneous. {\sc Case 2.} Assume that $k\ge 3$. Then $D=2$ is not possible. So, if $D=3$ then $\G^\circ$ is a complete graph $K_{\kappa+1}$ for some $\kappa\ge 3$, which, by Proposition~\ref{km}, implies that $\G$ is $2$-$X^\circ$-homogeneous. If $D\geq 4$, then by Theorem~\ref{Db}, for every $y\in \G_2(x)$ and $i$ $(2\leq i\leq D-2)$ the set $\G_{i,i}(x,y)$ is empty. Thus, for all $i$ $(2\le i\le D-1)$ and for all $x\in X^\circ$, $y\in\G_2(x)$ and $z\in\G_{i,i}(x,y)$, the number $|\G_{i-1}(z)\cap\G_{1,1}(x,y)|$ equals $0$. The result follows.
\end{proof}

\medskip
Note that if $k'=2$ then $\left|\G_2(x)\right|=\deg(x)$. Moreover, as we will see in Corollary~\ref{Dy}, for the general case when $D=3$, we have that graph $\G$ is $2$-$Y$-homogeneous if and only if $\left|\G_2(x)\right|=\deg(x)$.

\begin{remark}{\rm
Assume that $\G$ is a $(Y,Y')$-distance-biregular graph with $k'=2$. Since $c'_2+b'_2=k'$, we have $c_2'=1$. In Section~\ref{Eb} we study distance-biregular graphs with $k'\ge 3$ and $c_2'=1$.
}\end{remark}

%%%%%%%%%%%%%%%%%%%%%%%%%%%%%%%%%%%
%%%%%%%%%%%%%%%%%%%%%%%%%%%%%%%%%%%
%%%%%%%%%%%%%%%%%%%%%%%%%%%%%%%%%%%
%%%%%%%%%%%%%%%%%%%%%%%%%%%%%%%%%%%
%%%%%%%%%%%%%%%%%%%%%%%%%%%%%%%%%%%
%%%%%%%%%%%%%%%%%%%%%%%%%%%%%%%%%%%
%%%%%%%%%%%%%%%%%%%%%%%%%%%%%%%%%%%
%%%%%%%%%%%%%%%%%%%%%%%%%%%%%%%%%%%

\section{The scalars $\boldsymbol{\Delta_i}$, part I}
\label{D1}

In this section we define certain scalars $\Delta_i$ $(2\le i\le \min\left\lbrace D-1, D'-1 \right\rbrace)$, which can be computed from the intersection array of a given distance-biregular graph. These scalars play an important role, since from their values we can decide if a given distance-biregular graph is (almost) $2$-$Y$-homogeneous or not. The main ideas for both defining the scalar $\Delta_i$ and proving some related results are taken from the theory of bipartite distance-regular graphs. See for example \cite{BC,MMP1,MMP2,MP,PS,SP} for more details.

\begin{definition}{\rm
Let $\G$ denote a distance-biregular graph $\G$ with color partitions $(Y,Y')$, pick $i$ $(1\leq i \leq \min\left\lbrace D-1, D'-1 \right\rbrace)$, and define the scalar $\Delta_i=\Delta_i(Y)$ in the following way:
$$
\Delta_{i}=
\left\{\begin{array}{rl}
\displaystyle(b_{i-1}-1)(c_{i+1}-1)-p^i_{2,i}(c_2'-1) & \mbox{if } i \mbox{ is even,}\\
\displaystyle (b'_{i-1}-1)(c'_{i+1}-1)-p^i_{2,i}(c_2'-1) & \mbox{if } i \mbox{ is odd.} 
\end{array}\right. 
$$
}\end{definition}

\begin{remark}{\rm
If $D$ is odd, since for any $x\in Y$, $\G_D(x)\subseteq Y'$, we have $D'\ge D$, and with it for odd $D$, $\min\left\lbrace D-1, D'-1 \right\rbrace=D-1$. This yields that, if $D$ is even and $D'=D-1$ then the numbers $\Delta_i$ are defined for $2\le i\le D-2$, while in all other cases the scalars are defined for $2\le i\le D-1$. We already mentioned that, in the case when $D'=D-1$ we have $\gamma_{D-1}=c_2$ (see the proof of Corollary~\ref{Ek}), which yields that we do not need the definition of $\Delta_{D-1}$ for this case. Since we are working with (almost) $2$-$Y$-homogeneous DBG, we are not interested in the case when $i=D$.
}\end{remark}

The next inequality will be important and useful later.

\begin{lemma}
\label{4t}
Let $s_i,t_i\in\RR$ $(1\le i\le n)$ and $c\in \RR$. If $s_i+t_i=c$ for all $i$ $(1\le i\le n)$ then
	$$
	\left(\sum_{i=1}^n s_i \right)
	\left(\sum_{i=1}^n t_i \right)\ge 
	n\left(\sum_{i=1}^n s_it_i \right).
	$$
	Moreover, the equality holds if and only if all the numbers $s_i \; (1 \leq i \leq n)$ are equal to their arithmetic mean. 
\end{lemma}

\begin{proof}
The idea for the proof we found in {{\rm\cite[Chapter~4]{CZ}}}.

Let $\ol{s}$ denote the average value of the numbers $s_i$ for $1\le i\le n$. Note that $\sum_{i=1}^n \ol{s}^2=\ol{s} \sum_{i=1}^n s_i$, and compute $\sum_{i=1}^n (s_i -\ol{s})^2$.
\end{proof}

\begin{lemma}
\label{Da}
With reference to Notation~\ref{GN}, let $\G$ denote a distance-biregular graph with $k'\ge 3$ and $D\ge 3$. Then the number $\Delta_{i}$ is nonnegative for every integer $i \; (1\leq i \leq \min\left\lbrace D-1, D'-1 \right\rbrace)$. 
\end{lemma}

\begin{proof}
The idea for the proof we found in \cite{BC}. Pick $x \in Y$ and fix an integer $i \; (1\leq i \leq \min\left\lbrace D-1,D'-1\right\rbrace)$. Let $z \in \G_i(x)$, consider the set $\G_{2,i}(x,z)$ and note that $\left|\G_{2,i}(x,z) \right|=p^i_{2,i}$. By \eqref{Dj}, \eqref{Di} and \eqref{Dl}, the scalars $p^{1}_{2,1}$ and $p^{D-1}_{2,D-1}$ are nonzero. Moreover, if $D\ge 4$, since $k'>2$, Theorem~\ref{Db} yields that the scalar $p^2_{j,j}$ is nonzero for $2\le j\le D-2$. By \eqref{Di}, the scalar $p^j_{2,j}$ is nonzero for all $j$, so the set $\G_{2,i}(x,z)$ is not empty. For any $y\in\G_{2,i}(x,z)$, $|\G_{1,1,i-1}(x,y,z)|+|\G_{1,1,i+1}(x,y,z)|=c_2$, so Lemma~\ref{4t} yields
\begin{equation}
\label{Dp}
\begin{split}
\left(\sum_{y \in \G_{2,i}(x,z)}|\G_{1,1,i-1}(x,y,z)|\right)
\left( \sum_{y \in \G_{2,i}(x,z)} |\G_{1,1,i+1}(x,y,z)|\right)\geq
\qquad\qquad\qquad\qquad\\
\qquad\qquad\qquad\qquad\ge|\G_{2,i}(x,z)| \left( \sum_{y \in \G_{2,i}(x,z)} |\G_{1,1,i-1}(x,y,z)|\cdot|\G_{1,1,i+1}(x,y,z)|\right).
\end{split}
\end{equation}
The sum $\sum_{y \in \G_{2,i}(x,z)} |\G_{1,1,i-1}(x,y,z)|$ represents the number of ordered pairs $(y,w)$ such that $y\in\G_{2,i}(x,z)$ and $w\in\G_{1,1,i-1}(x,y,z)$. We can compute this number in another way, that is, by counting first the number of vertices $w\in\G_{1,i-1}(x,z)$, and for every such $w$ counting the number of vertices $y\in\G_{2,1,i}(x,y,z)$. For that purpose we consider separately the cases depending on the parity of $i$, because for example, if $i$ is even then $z \in Y$, and if $i$ is odd then $z \in Y'$. If $i$ is even, since $x\in\G_i(z)$, $x$ has exactly $c_i$ neighbours $w$ at distance $i-1$ from $z$, and if $i$ is odd, $x$ has exactly $c_i'$ neighbours $w$ at distance $i-1$ from $z$. Next, for a given $w\in\G_{1,i-1}(x,z)$ we count the number of vertices $y\in\G_{2,i}(x,z)$ which are neighbours of $w$. Note that every vertex in $\G_{1,i}(w,z)$ is either $x$ or is at distance $2$ from $x$. Since $w\in\G_{i-1}(z)$ we therefore have $b_{i-1}-1$ possibilities for the choice of $y$ if $i$ is even or $b'_{i-1}-1$ possibilities for the choice of $y$ if $i$ is odd. Thus
\begin{equation}
\label{Dn}
\sum_{y \in \G_{2,i}(x,z)} |\G_{1,1,i-1}(x,y,z)|=
\left\{\begin{array}{rl}
c_i(b_{i-1}-1) & \mbox{if } i \mbox{ is even,}\\
c'_i(b'_{i-1}-1)  & \mbox{if } i \mbox{ is odd.}
\end{array}\right.
\end{equation}
Using the same technique as above, it is routine to compute
\begin{equation}
\label{Dr}
\sum_{y\in\G_{2,i}(x,z)} |\G_{i+1}(z)\cap\G_{1,1}(x,y)|=
\left\{\begin{array}{rl}
\displaystyle b_i(c_{i+1}-1) & \mbox{if } i \mbox{ is even,}\\
\displaystyle b'_i(c'_{i+1}-1)  & \mbox{if } i \mbox{ is odd,} 
\end{array}\right.
\end{equation}
and
\begin{equation}
\label{Ds}
\sum_{y\in\G_{2,i}(x,z)}|\G_{1,1,i-1}(x,y,z)|\cdot{|\G_{1,1,i+1}(x,y,z)|}=
\left\{\begin{array}{rl}
c_ib_i(c'_{2}-1) & \mbox{if } i \mbox{ is even, }\\
c_i'b'_i(c'_{2}-1)  & \mbox{if } i \mbox{ is odd. }
\end{array}\right. 
\end{equation}
The result follows.
\end{proof}

Note that, in the proof of Lemma~\ref{Da}, in two different ways we counted the number of ordered pairs $(y, w)$ which are at certain distances from three fixed vertices $x$, $y$ and $z$. For a ``different'' and more advanced technique of counting the number of ordered tuples we recommend \cite{WMS, NP}. Next we consider the case when the equality in \eqref{Dp} holds (the idea for Proposition~\ref{Dm} we found in \cite{BC}, where the author had studied bipartite distance-regular graphs).

\begin{proposition}
\label{Dm}
With reference to Notation~\ref{GN}, let $\G$ denote a $(Y,Y')$-distance-biregular graph with $k'\ge 3$ and $D\ge 3$. For any $i \; (2\leq i \leq \min\left\lbrace D-1, D'-1 \right\rbrace)$, the following are equivalent.
\begin{enumerate}[label=(\roman*),font=\rm]
\item The scalar $\Delta_i=0$. 
\item For all $x\in Y$, $y\in\G_2(x)$ and $z\in\G_{i,i}(x,y)$, the number $|\G_{1,1,i-1}(x,y,z)|$ is independent of the choice of $z$. In addition, we have
\begin{eqnarray}
\gamma_i:=|\G_{1,1,i-1}(x,y,z)|=
\left\{\begin{array}{rl}
c_i(b_{i-1}-1)/p^i_{2,i} & \mbox{if } i \mbox{ is even, }\\
c'_i(b'_{i-1}-1)/p^i_{2,i}  & \mbox{if } i \mbox{ is odd. }
\end{array}\right. \nonumber 
\end{eqnarray} 
\item There exist $x \in Y$ and $z\in\G_i(x)$ such that for all $y\in\G_{2,i}(x,z)$ the number 
$
|\G_{1,1,i-1}(x,y,z)|
$
is independent of the choice of $y$.
\end{enumerate}
\end{proposition}

\begin{proof} 
Pick some integer $i \; (2\leq i \leq \min\left\lbrace D-1, D'-1 \right\rbrace )$. As in the proof of Lemma~\ref{Da}, it can be shown that the scalar $p^{i}_{2,i}$ is nonzero. 

(i)$\Ra$(ii) In the proof of Lemma~\ref{Da}, using Lemma~\ref{4t}, we showed that for any $x\in Y$ and $z\in\G_i(x)$, if $i$ is even then
$$
(b_{i-1}-1)(c_{i+1}-1)\ge p^i_{2,i}(c_2'-1),
$$
and if $i$ is odd then
$$
(b'_{i-1}-1)(c'_{i+1}-1)\ge p^i_{2,i}(c_2'-1).
$$
Note that equality holds if and only if all the numbers $|\G_{1,1,i-1}(x,y,z)|$ (where ${y \in \G_{2,i}(x,z)}$) are equal to their arithmetic mean, that is for every $y \in \G_{2,i}(x,z)$,
$$
|\G_{1,1,i-1}(x,y,z)|=\frac{1}{|\G_{2,i}(x,z)|} \sum_{w\in \G_{2,i}(x,z)} |\G_{1,1,i-1}(x,w,z)|.
$$
Since by assumption $\Delta_i=0$ the result follows from \eqref{Dn}.

\smallskip
(ii)$\Ra$(iii) Immediate.

\smallskip
(iii)$\Ra$(i) Pick $x\in Y$ and let $z\in \G_i(y)$. Since the number $p^i_{2,i}$ is nonzero, the set $\G_{2,i}(x,z)$ is nonempty. Assume that for all $y\in\G_{2,i}(x,z)$ the number $|\G_{1,1,i-1}(x,y,z)|$ is independent of $y$. Since these numbers do not depend on the choice of $y\in\G_{2,i}(x,z)$, the numbers $|\G_{1,1,i-1}(x,y,z)|$ are all equal to their average value, that is, for $y\in \G_{2,i}(x,z)$,
$$
|\G_{1,1,i-1}(x,y,z)|\cdot{|\G_{2,i}(x,z)|}=
\sum_{w \in \G_{2,i}(x,z)} |\G_{1,1,i-1}(x,w,z)|.
$$ 
Observe $|\G_{1,1,i-1}(x,y,z)|+|\G_{1,1,i+1}(x,y,z)|=c_2$ for all $y\in \G_{2,i}(x,z)$. Therefore, by Lemma \ref{4t}, equality holds in \eqref{Dp}.	Now the equality $\Delta_i=0$ follows from \eqref{Dn}, \eqref{Dr} and \eqref{Ds}.
\end{proof}

%%%%%%%%%%%%%%%%%%%%%%%%%%%%%%%%%%%%%%%
%%%%%%%%%%%%%%%%%%%%%%%%%%%%%%%%%%%%%%%
%%%%%%%%%%%%%%%%%%%%%%%%%%%%%%%%%%%%%%%
%%%%%%%%%%%%%%%%%%%%%%%%%%%%%%%%%%%%%%%
%%%%%%%%%%%%%%%%%%%%%%%%%%%%%%%%%%%%%%%
%%%%%%%%%%%%%%%%%%%%%%%%%%%%%%%%%%%%%%%

\section{Positivity of the scalars $\boldsymbol{\gamma_i}$ and equitable partition}
\label{oH}

Let $\G$ denote a $(Y,Y')$-distance-biregular graph with $k'\ge 3$ and $D\ge 3$. In this section we show that for $2$-$Y$-homogeneous $\G$, for every integer $i$  $(1\le i\le \min\left\lbrace D-1, D'-1\right\rbrace )$ and for all $x\in Y$, $y \in \G_2(x)$ and $z\in \G_{i,i}(x,y)$ the number $\gamma_i=|\G_{1,1,i-1}(x,y,z)|$ is nonzero. We also show that the collection of all the non-empty sets $\G_{i,j}(x,y)$ $(0 \leq i, j \leq D)$  is an equitable partition (for any $x\in Y$ and $y\in\G_2(x)$) if and only if $\G$ is $2$-$Y$-homogeneous.

\begin{lemma}
\label{nH}
With reference to Notation~\ref{GN}, let $\G$ denote a $(Y,Y')$-distance-biregular graph with $k'\ge 3$ and $D\ge 3$. Pick $x\in Y$, $y\in\G_2(x)$ and an integer $i$  $(2\le i\le \min\left\lbrace D-1, D'-1\right\rbrace )$. Then, for $z\in\G_{i,i}(x,y)$ the following holds:
\begin{equation}
\label{pH}
\sum_{v\in\G_{i-1,i-1,1}(x,y,z)} |\G_{1,1.i-2}(x,y,v)|=c_{i-1}'|\G_{1,1,i-1}(x,y,z)|.	
\end{equation}
\end{lemma}

\begin{proof}
We count, in two different ways, the number $N$ of ordered pairs $(u,v)$ with $\partial(u,v)=i-2$, $u \in \G_{1,1,i-1}(x,y,z)$ and $v\in \G_{i-1,i-1,1}(x,y,z)$. Observe that for every fixed vertex $v\in \G_{i-1,i-1,1}(x,y,z)$ we have $\left|\G_{1,1,i-2}(x,y,v)\right|$ possible choices for $u$. For any such vertex $u$, $\partial(u,z)=i-1$. So, 
$$
N=\sum_{v\in\G_{i-1,i-1,1}(x,y,z)} |\G_{1,1,i-2}(x,y,v)|.
$$
We next fix $u \in \G_{1,1,i-1}(x,y,z)$ and observe that, for such $u$ we have $\left| \G_{i-2,1}(u,z)\right|$ possible choices for $v$. By the triangle inequality of distances in triangles $xvz$, $xuv$, $uvy$ and $zvy$, it turns out that $v\in \G_{i-1,i-1}(x,y)$. Note also that $u \in Y'$ and $z\in \G_{i-1}(u)$. Since $\G$ is distance-biregular, $N=c'_{i-1}|\G_{1,1,i-1}(x,y,z)|$. The claim follows. 
\end{proof}

\begin{theorem}
\label{mH}
With reference to Notation~\ref{GN}, let $\G$ denote a $(Y,Y')$-distance-biregular graph with $k'\ge 3$ and $D\ge 3$. Suppose that $\G$ is $2$-$Y$-homogeneous. Then, for every  integer $i$  $(1\le i\le \min\left\lbrace D-1, D'-1\right\rbrace )$ and for all $x\in Y$, $y \in \G_2(x)$ and $z\in \G_{i,i}(x,y)$ the number $\gamma_i=|\G_{1,1,i-1}(x,y,z)|$ is nonzero. 
\end{theorem}

\begin{proof} 
Since $\G$ is $2$-$Y$-homogeneous, for all $i$  $(1\le i\le \min\left\lbrace D-1, D'-1\right\rbrace )$ and for all $x\in Y$, $y \in \G_2(x)$ and $z\in \G_{i,i}(x,y)$ the number $\gamma_i=|\G_{1,1,i-1}(x,y,z)|$ does not depend on the choice of $x, y$ and $z$. It follows from the definition that $\gamma_1=1$. For $i\geq 2$, we will proceed by contradiction. Suppose first $\gamma_2=0$. The number $p^2_{2,2}$ is not zero by \eqref{Dj} and Theorem~\ref{Db}. Then, since $c_2$ is a positive integer, by Proposition \ref{Dm} we have $b_1=1$. Therefore, $k'=b_1+c_1=2$ which is a contradiction. If $D=3$ we are done. Otherwise, assume that there exists $\ell$ $(3\leq \ell \leq \min\left\lbrace D-1, D'-1 \right\rbrace )$ such that $\gamma_\ell=0$. Without loss of generality, we can pick an integer $i$ $(3\leq i \leq \min\left\lbrace D-1, D'-1 \right\rbrace )$ such that $\gamma_{i-1}\ne 0$ and $\gamma_i=0$. By Theorem~\ref{Db} and \eqref{Di}, $p^i_{2,i}$ is not zero. As $c_i$ and $c'_i$ are positive, by Proposition \ref{Dm} either $b_{i-1}=1$ (if $i$ is even) or $b'_{i-1}=1$ (if $i$ is odd). Now, by Lemma~\ref{nH}, for every $x\in Y$ with $y\in\G_2(x)$ and $z \in \G_{i,i}(x,y)$, 
$$
c_{i-1}' \gamma_i = \gamma_{i-1} |\G_{i-1,i-1,1}(x,y,z)|
$$
and since the number $\gamma_{i-1}$ is nonzero, we have $|\G_{i-1,i-1,1}(x,y,z)|=0$. We next consider two cases.
	
{\sc Case 1.} Suppose that $i$ is even.  By Lemma \ref{4w} we have $b_{i-1}\ge b_{i+1}$ and so $b_{i+1}=1$. Now, from \eqref{id6}, for every vertex $z\in\G_{i+1,i-1}(x,y)\cup\G_{i-1,i+1}(x,y)$ we have $|\G_{i,i,1}(x,y,z)|=b_{i-1}-b_{i+1}=0$. Hence,  $|\G_{i-1,i+1,1}(x,y,z)|=|\G_{i+1,i-1,1}(x,y)|=0$ for every $z\in \G_{i,i}(x,y)$.
	This yields the set $\G_{i,i}(x,y)$ is empty, contradicting Theorem~\ref{Db}. 
	
{\sc Case 2.} Assume next that $i$ is odd. Since $b_{i-1}'=1$, by Lemma~\ref{4w} we have $b_i=1$. Recall that arbitrary $z\in\G_{i,i}(x,y)$ does not have neighbours in $\G_{i-1,i-1}(x,y)$. Considering the intersection diagram of rank $2$ (see Figure~\ref{02}) this implies that $z\in\G_{i,i}(x,y)$ has exactly $b_i-c_i$ neighbours in $\G_{i+1,i+1}(x,y)$, which yields $c_i=1$. Thus, $k'=b_i+c_i=2$ contradicting our assumption $k'\ge 3$. 
	
The claim follows. 
\end{proof}

\begin{proposition}
\label{Ex}
With reference to Notation~\ref{GN}, let $\G$ denote a $(Y,Y')$-distance-biregular graph with $k'\ge 3$ and $D\ge 3$. The following are equivalent.
\begin{enumerate}[label=(\roman*),font=\rm]
\item $\G$ is $2$-$Y$-homogeneous. 
\item For all integers $i$ $(1\le i\le\min\lbrace D-1,D'-1\rbrace)$ and for all $x\in Y$, $y\in\G_2(x)$ and $z\in\G_{i,i}(x,y)$, the scalar $\delta_i=|\G_{i-1,i-1,1}(x,y,z)|$ is nonzero and it is independent of the choice of $x$, $y$, and $z$.
\end{enumerate}
\end{proposition}

\begin{proof} 	
The idea for the proof we found in \cite[Theorem 16]{BC}, where the author studied bipartite distance-regular graphs.

(i)$\Ra$(ii) Assume that $\G$ is $2$-$Y$-homogeneous. By Theorem~\ref{mH}, for every $i$ $(1\le i\le \min\left\lbrace D-1, D'-1\right\rbrace )$ and for all $x\in Y$, $y \in \G_2(x)$ and $z\in \G_{i,i}(x,y)$, $\gamma_i=|\G_{1,1,i-1}(x,y,z)|$  is nonzero and does not depend on the choice of $x$, $y$ and $z$. Moreover, by Lemma  \ref{nH}, $$c_{i-1}' \gamma_i = \gamma_{i-1} |\G_{i-1,i-1,1}(x,y,z)|.$$
Since $\gamma_{i}, \gamma_{i-1}, c_{i-1}'\;  (2\leq i \leq \min\left\lbrace D-1, D'-1 \right\rbrace )$ are nonzero, we have $|\G_{i-1,i-1,1}(x,y,z)|$ is nonzero and does not depend on the choice of $x$, $y$ and $z$.

\medskip
(ii)$\Ra$(i) Assume that for all integers $h$ $(1\le h\le\min\lbrace D-1,D'-1\rbrace)$ and for all $x\in Y$, $y\in\G_2(x)$, $z\in\G_{h,h}(x,y)$, the scalar $\delta_h=|\G_{1}(z)\cap\G_{h-1,h-1}(x,y)|$ is independent of the choice of $y$, $z$, and it is nonzero. Note that $\gamma_1=1$ and $\gamma_2=\delta_2$. Setting $i=2,3,\ldots,\ell$ (where $\ell\le\min\lbrace D-1,D'-1\rbrace$) in \eqref{pH}, and using mathematical induction we get
$$
\underbrace{|\G_{\ell-1}(z)\cap\G_{11}(x,y)|}_{=\gamma_\ell}=\frac{1}{c'_{\ell-1}} \gamma_{\ell-1}\delta_\ell,
$$
and with it
$$
\gamma_i=\frac{\delta_2\delta_3\cdots\delta_i}{c_2'c_3'\cdots c_{i-1}'}\qquad(1\le i\le\min\lbrace D-1,D'-1\rbrace).
$$
The result follows.
\end{proof}

\begin{remark}{\rm
\label{Ey}
Let $\G$ denote a $(Y,Y')$-distance-biregular graph with $k'\ge 3$ and $D\ge 3$. Fix $i$ $(2\leq i \leq \min\left\lbrace D-1, D'-1 \right\rbrace)$. From the results of this paper we do not know if the following two claims are equivalent.
\begin{enumerate}[label=(\roman*),font=\rm]
\item The scalar $\Delta_i=0$.
\item There exist $x \in Y$ and $y\in\G_2(x)$ such that for all $z\in\G_{i,i}(x,y)$ the number $|\G_{i-1,i-1,1}(x,y,z)|$ is independent of the choice of $z$.
\end{enumerate}
}\end{remark}

\begin{corollary}
\label{Fg}
With reference to Notation~\ref{GN}, let $\G$ denote a $(Y,Y')$-distance-biregular graph with $k'\ge 3$ and $D\ge 3$.  The following are equivalent.
\begin{enumerate}[label=(\roman*),font=\rm]
\item The collection of all non-empty sets $\G_{i,j}(x,y)$ $(0\le i,j\le D)$ is an equitable partition of $\G$ (for any $x\in Y$ and $y\in\G_2(x)$). 
\item $\G$ is $2$-$Y$-homogeneous.
\end{enumerate}
\end{corollary}

\begin{proof}
(i)$\Ra$(ii) Immediate.

(ii)$\Ra$(i) Immediate from \eqref{id6}, Theorem~\ref{Db} and Proposition~\ref{Ex} (see Figure~\ref{02}).
\end{proof}

%%%%%%%%%%%%%%%%%%%%%%%%%%%%%%%%%%%%%%%%%%%%
%%%%%%%%%%%%%%%%%%%%%%%%%%%%%%%%%%%%%%%%%%%%
%%%%%%%%%%%%%%%%%%%%%%%%%%%%%%%%%%%%%%%%%%%%
%%%%%%%%%%%%%%%%%%%%%%%%%%%%%%%%%%%%%%%%%%%%
%%%%%%%%%%%%%%%%%%%%%%%%%%%%%%%%%%%%%%%%%%%%
%%%%%%%%%%%%%%%%%%%%%%%%%%%%%%%%%%%%%%%%%%%%
%%%%%%%%%%%%%%%%%%%%%%%%%%%%%%%%%%%%%%%%%%%%
%%%%%%%%%%%%%%%%%%%%%%%%%%%%%%%%%%%%%%%%%%%%

\section{The scalars $\boldsymbol{\Delta_i}$, part II}
\label{Ev}

The main result of this section is somehow unexpected: for a fixed $i$, the scalar $\Delta_i=0$ if and only if there exist $x \in Y$ and $y\in\G_2(x)$ such that for all $z\in\G_{i,i}(x,y)$ the number $|\G_{1,1,i-1}(x,y,z)|$ is independent of the choice of $z$. The proof is tedious and time consuming. The main technique is counting the number of ordered pairs in two different ways. As a corollary we get $\G$ is almost $2$-$Y$-homogeneous if and only if $\Delta_i=0$ $(2\le i\le D-2)$. We start with Lemma~\ref{Ei}, which we use in the proof of Theorem~\ref{Eg}.

\begin{lemma}
\label{Ei}
With reference to Notation~\ref{GN}, let $\G$ denote a $(Y,Y')$-distance-biregular graph with $k'\ge 3$ and $D\ge 3$. For any $i \; (2\leq i \leq \min\left\lbrace D-1, D'-1 \right\rbrace)$, the following holds.
\begin{enumerate}[label=(\roman*),font=\rm]
\item Pick $x\in Y$, $y\in\G_2(x)$ and $w\in\G_{1,1}(x,y)$. Then
$$
|\G_{i,i,i-1}(x,y,w)|=
\left\{\begin{array}{cl}
\frac{c_ik_i(b_{i-1}-1)}{b_0b_1} & \mbox{if } $i$ \mbox{ is even,}\\
\frac{c'_ik_i(b'_{i-1}-1)}{b_0b_1}  & \mbox{if } i \mbox{ is odd.}\\
\end{array}\right.
$$
\item Pick $x\in Y$, $y\in\G_2(x)$ and $w\in\G_{1,1}(x,y)$. Then
$$
|\G_{i,i,i+1}(x,y,w)|=
\left\{\begin{array}{cl}
\frac{b_ik_i(c_{i+1}-1)}{b_0b_1} & \mbox{if } i \mbox{ is even,}\\
\frac{b'_ik_i(c'_{i+1}-1)}{b_0b_1}  & \mbox{if } i \mbox{ is odd.}
\end{array}\right.
$$
\item Pick $x\in Y$, $y\in\G_2(x)$ and $u,v\in\G_{1,1}(x,y)$. If $\partial(u,v)=2$ then
$$
|\G_{i-1,i+1}(u,v)|=
|\G_{i,i,i-1,i+1}(x,y,u,v)|=
\left\{\begin{array}{cl}
\frac{k_ib_ic_i}{b_0b'_1} & \mbox{if } i \mbox{ is even, }\\
\frac{k_ib'_ic'_{i}}{b_0b'_1}  & \mbox{if } i \mbox{ is odd. }
\end{array}\right. 
$$
\end{enumerate}
\end{lemma}

\begin{proof}
(i) We have $|\G_{i,i-1}(x,w)|=|\G_{i,i-2,i-1}(x,y,w)| + |\G_{i,i,i-1}(x,y,w)|
=|\G_{i,i-2}(x,y)|+|\G_{i,i,i-1}(x,y,w)|$, that is
\begin{equation}
\label{Fa}
|\G_{i,i,i-1}(x,y,w)|=|\G_{i,i-1}(x,w)|-|\G_{i,i-2}(x,y)|.
\end{equation}
By \eqref{r2}, $|\G_{i,i-1}(x,w)|=\frac{b_1b_2\cdots b_{i-1}}{c_1'c_2'\cdots c_{i-1}'}$, and by \eqref{am}, $|\G_{i,i-2}(x,y)|=\frac{b_2b_3\cdots b_{i-1}}{c_1c_2\cdots c_{i-2}}$. Now we consider two cases.

{\sc Case 1.} Assume that $i$ is even. Because $i-1$ is odd, by \eqref{Dk}, \eqref{r2}, \eqref{am} and \eqref{Fa} we have 
$$
|\G_{i,i,i-1}(x,y,w)|=\frac{b_2\cdots b_{i-1}}{c_1c_2\cdots c_{i-2}}
\cdot\frac{b_1-c_{i-1}}{c_{i-1}}=
\frac{k_ic_i}{b_0b_1}(b_1-c_{i-1}).
$$
Since $k'=b_{i-1}+c_{i-1}=b_1+c_1$ we have $b_1-c_{i-1}=b_{i-1}-1$, and the result follows.

{\sc Case 2.} Assume that $i$ is odd. By \eqref{Dk} and \eqref{r2},
$$
|\G_{i,i,i-1}(x,y,w)|=\frac{b_2\cdots b_{i-1}}{c_1c_2\cdots c_{i-2}}
\cdot\frac{b_1-c'_{i-1}}{c'_{i-1}}.
$$
Since $k'=b'_{i-1}+c'_{i-1}=b_1+c_1$ we have $b_1-c'_{i-1}=b'_{i-1}-1$, and the result follows.

\medskip
(ii) Recall that $|\G_i(x)|=k_i$ and pick $z\in \G_{i,i+1}(x,w)$. By the triangle inequality of distances in triangles $zxy$ and $zwy$, we have $\partial(y,z) \in \left\lbrace i, i+2 \right\rbrace$. Even more, $\G_{i,i+2, i+1}(x,y,w)=\G_{i,i+2}(x,y)$. This yields \begin{eqnarray}
\label{p33}
|\G_{i,i, i+1}(x,y,w) |=|\G_{i, i+1}(x,w) |-|\G_{i, i+2}(x,y) |.
\end{eqnarray}
Assume $D'\geq D$. Then, it is easy to see that the result holds for $i=D-1$. So, suppose that $2\leq i\leq D-2$. By \eqref{Ej}, \eqref{r4} and \eqref{am}, we have 
\begin{equation}
\label{p3}
|\G_{i,i+1}(x,w)|=\frac{b'_1b'_2\cdots b'_{i-1}}{b_1b_2\cdots b_{i-1}}\cdot{\frac{b'_ik_i}{b_0}}\qquad\mbox{and}\qquad
|\G_{i,i+2}(x,y)|= \frac{b_{i+1}b_ik_i}{b_0b_1}. 
\end{equation}
We next consider two cases. Assume first that $i$ is even. Then, $b_1-b_{i+1}=c_{i+1}-1$. Note that $b_{i}\neq0$ and by \eqref{Du}, since $i+1$ is odd, $b_1b_2\cdots b_{i}=b'_1b'_2\cdots b'_{i}$. Suppose now that $i$ is odd. Then $b_1-b'_{i+1}=c'_{i+1}-1$. We also have the products $b_1b_2\cdots b_{i-1}=b'_1b'_2\cdots b'_{i-1}$ and $b_{i+1}b_i=b'_{i+1}b_i'$. The result follows immediately from \eqref{p33}, \eqref{p3} and the above comments. 

\medskip
(iii) Note that $u,v \in Y'$ and $\partial(u,v)=2$. Similarly as in \eqref{ap}, it is routine to get $|\G_{i-1,i+1}(u,v)|=\frac{b_2'b_3'\cdots b_i'}{c_1'c_2'\cdots c'_{i-1}}=\frac{1}{b_1'}\frac{b_1'b_2'b_3'\cdots b_i'}{c_1'c_2'\cdots c'_{i-1}}$. Splitting into two cases, depending on the parity of $i$, the result now immediately follows from \eqref{Ej}, \eqref{Dk} and \eqref{Du}.
\end{proof}

\begin{theorem}
\label{Eg}
With reference to Notation~\ref{GN}, let $\G$ denote a $(Y,Y')$-distance-biregular graph with $k'\ge 3$ and $D\ge 3$. For any $i \; (2\leq i \leq \min\left\lbrace D-1, D'-1 \right\rbrace)$, the following are equivalent.
\begin{enumerate}[label=(\roman*),font=\rm]
\item The scalar $\Delta_i=0$.
\item There exist $x \in Y$ and $y\in\G_2(x)$ such that for all $z\in\G_{i,i}(x,y)$ the number $|\G_{1,1,i-1}(x,y,z)|$ is independent of the choice of $z$.
\end{enumerate}
\end{theorem}

\begin{proof}
Pick $x \in Y$ and $y \in \G_2(x)$. Recall that, since $k'\ge 3$, by Theorem~\ref{Db}, $p^2_{i,i}$ is nonzero. Assume that for all $z\in\G_{i,i}(x,y)$ the number $|\G_{1,1,i-1}(x,y,z)|$ does not depend on the choice of $z$. With it the numbers $|\G_{1,1,i-1}(x,y,z)|$ are all equal to their average value. Since $|\G_{1,1,i-1}(x,y,z)|+|\G_{1,1,i+1}(x,y,z)|=c_2$, Lemma \ref{4t} yields
\begin{equation}
\label{Ec}
\begin{split}
\left(\sum_{z \in \G_{i,i}(x,y)}|\G_{1,1,i-1}(x,y,z)|\right)
\left( \sum_{z \in \G_{i,i}(x,y)} |\G_{1,1,i+1}(x,y,z)|\right)=
\qquad\qquad\qquad\qquad\\
\qquad\qquad\qquad\qquad=|\G_{i,i}(x,y)| \left( \sum_{z \in \G_{i,i}(x,y)} |\G_{1,1,i-1}(x,y,z)|\cdot|\G_{1,1,i+1}(x,y,z)|\right).
\end{split}
\end{equation}
By \eqref{Du} and \eqref{Di}, \eqref{Ec} becomes
\begin{equation}
\label{Eh}
\begin{split}
\frac{b_0b_1}{c_2k_i}\left(\sum_{z \in \G_{i,i}(x,y)}|\G_{1,1,i-1}(x,y,z)|\right)
\frac{b_0b_1}{c_2k_i}\left( \sum_{z \in \G_{i,i}(x,y)} |\G_{1,1,i+1}(x,y,z)|\right)=
\qquad\qquad\qquad\qquad\\
\qquad\qquad\qquad\qquad=p^i_{2,i}\frac{(c_2'-1)b_0b'_1}{c_2(c_2-1)k_i}
\left( \sum_{z \in \G_{i,i}(x,y)} |\G_{1,1,i-1}(x,y,z)|\cdot|\G_{1,1,i+1}(x,y,z)|\right)
\end{split}
\end{equation}
(note that we used $b_1=\frac{b_1'(c_2'-1)}{c_2-1}$).

As in the proof of Lemma~\ref{Da}, for a fixed $x\in Y$ and $y\in\G_2(x)$, counting the number of ordered pairs $(z,w)$ (where $z\in\G_{i,i}(x,y)$ and $w\in\G_{1,1,i-1}(x,y,z)$) in two different ways, using Lemma~\ref{Ei}(i) it follows
\begin{equation}
\label{Ed}
\sum_{z\in\G_{i,i}(x,y)} |\G_{1,1,i-1}(x,y,z)|=
\left\{\begin{array}{rl}
\frac{c_2k_ic_i(b_{i-1}-1)}{b_0b_1} & \mbox{if } i \mbox{ is even,}\\
\frac{c_2k_ic'_i(b'_{i-1}-1)}{b_0b_1}  & \mbox{if } i \mbox{ is odd.}
\end{array}\right.
\end{equation}
Similarly, using Lemma~\ref{Ei}(ii), it is routine to compute
\begin{equation}
\label{Ee}
\sum_{z\in\G_{i,i}(x,y)}|\G_{1,1,i+1}(x,y,z)|=  
\left\{\begin{array}{cl}
\frac{c_2b_ik_i(c_{i+1}-1)}{b_0b_1} & \mbox{if } $i$ \mbox{ is even,}\\
\frac{c_2b'_ik_i(c'_{i+1}-1)}{b_0b_1}  & \mbox{if } i \mbox{ is odd,}
\end{array}
\right. 
\end{equation}
and for a fixed $x\in Y$ and $y\in\G_2(x)$, counting the numbers of triples $(u,v,z)$ (where $u,v,\in\G_{1,1}(x,y)$, $\partial(u,v)=2$, $z\in\G_{i,i,i-1,i+1}(x,y,u,v)$) in two different ways, we get
\begin{equation}
\label{Ef}
\sum_{z\in\G_{i,i}(x,y)}|\G_{1,1,i-1}(x,y,z)|\cdot{|\G_{1,1,i+1}(x,y,z)|}=
\left\{\begin{array}{cl}
\frac{c_2(c_2-1)k_ib_ic_i}{b_0b'_1} & \mbox{if } i \mbox{ is even, }\\
\frac{c_2(c_2-1)k_ib'_ic'_{i}}{b_0b'_1}  & \mbox{if } i \mbox{ is odd. }
\end{array}\right. 
\end{equation}
Since the integers $c_i$, $b_i$, $c'_i$ and $b'_i$ are positive, the equality $\Delta_i=0$ follows from \eqref{Eh}--\eqref{Ef} and the definition of $\Delta_i$.
\end{proof}

\begin{corollary}
\label{Ek}
With reference to Notation~\ref{GN}, let $\G$ denote a $(Y,Y')$-distance-biregular graph with $k'\ge 3$ and $D\ge 3$. Then the following are equivalent.
\begin{enumerate}[label=(\roman*),font=\rm]
\item $\Delta_i=0$ $(2\le i\le \min\{D-1,D'-1\})$.
\item For every $i$ $(2\le i\le D-1)$, there exist $x \in Y$ and $y\in\G_2(x)$ such that for all $z\in\G_{i,i}(x,y)$ the number $|\G_{1,1,i-1}(x,y,z)|$ is independent of the choice of $z$.
\item $\G$ is $2$-$Y$-homogeneous.
\end{enumerate}
\end{corollary}

\begin{proof}
Recall that $D' \in \{D-1, D, D+1\}$. 

\medskip
{\sc Case 1.} Assume that $D-1\ne D'$. Then $\min\left\lbrace D-1, D'-1 \right\rbrace=D-1$. The result now follows immediately from Proposition~\ref{Dm} and Theorem~\ref{Eg}.

\medskip
{\sc Case 2.} Assume that $D-1=D'$. This implies that $D$ is even and $\min\left\lbrace D-1, D'-1 \right\rbrace=D-2$. We prove that (i)$\Ra$(ii)$\Ra$(iii)$\Ra$(i). 

(i)$\Ra$(ii) Assume that $\Delta_i=0$ for $2\le i\le D-2$. By Proposition~\ref{Dm} and Theorem~\ref{Eg}, for all $x\in Y$ and $y\in\G_2(x)$ the number $\G_{1,1,i-1}(x,y,z)$ is independent of the choice of $z\in\G_{i,i}(x,y)$ $(2\le i\le D-2)$. Fix $x\in Y$ and $y\in\G_2(x)$. We need to show that $\G_{1,1,D-2}(x,y,z)$ is independent of $z\in\G_{D-1,D-1}(x,y)$. We show that $|\G_{1,1,D-2}(x,y,z)|=c_2$. Fix $z\in\G_{D-1,D-1}(x,y)$. For any $w\in\G_{2,2}(x,y)$ we have $\partial(z,w)\in\{D-2,D\}$. Note that, since $D$ is even,  $D-1$ is odd which yields that $z\in Y'$. Now, because $D'<D$ we have $\partial(z,w)=D-2$. The result follows.

(ii)$\Ra$(iii)$\Ra$(i) Follows immediately from Proposition~\ref{Dm} and Theorem~\ref{Eg}.
\end{proof}

\begin{corollary}
\label{El}
With reference to Notation~\ref{GN}, let $\G$ denote a $(Y,Y')$-distance-biregular graph with $k'\ge 3$ and $D\ge 3$. The following are equivalent.
\begin{enumerate}[label=(\roman*),font=\rm]
\item $\Delta_i=0$ $(2\le i\le D-2)$.
\item For every $i$ $(2\le i\le D-2)$, there exist $x \in Y$ and $y\in\G_2(x)$ such that for all $z\in\G_{i,i}(x,y)$ the number $|\G_{1,1,i-1}(x,y,z)|$ is independent of the choice of $z$.
\item $\G$ is almost $2$-$Y$-homogeneous.
\end{enumerate}
\end{corollary}

\begin{proof}
Since $D-2 \leq \min\left\lbrace D-1, D'-1 \right\rbrace$, the result follows immediately from Proposition~\ref{Dm} and Theorem~\ref{Eg}.
\end{proof}

%%%%%%%%%%%%%%%%%%%%%%%%%%%%%%%%%%%%%%%%%%%
%%%%%%%%%%%%%%%%%%%%%%%%%%%%%%%%%%%%%%%%%%%
%%%%%%%%%%%%%%%%%%%%%%%%%%%%%%%%%%%%%%%%%%%
%%%%%%%%%%%%%%%%%%%%%%%%%%%%%%%%%%%%%%%%%%%

%%%%%%%%%%%%%%%%%%%%%%%%%%%%%%%%%%%%%%%%%%%
%%%%%%%%%%%%%%%%%%%%%%%%%%%%%%%%%%%%%%%%%%%
%%%%%%%%%%%%%%%%%%%%%%%%%%%%%%%%%%%%%%%%%%%
%%%%%%%%%%%%%%%%%%%%%%%%%%%%%%%%%%%%%%%%%%%
%%%%%%%%%%%%%%%%%%%%%%%%%%%%%%%%%%%%%%%%%%%
%%%%%%%%%%%%%%%%%%%%%%%%%%%%%%%%%%%%%%%%%%%
%%%%%%%%%%%%%%%%%%%%%%%%%%%%%%%%%%%%%%%%%%%
%%%%%%%%%%%%%%%%%%%%%%%%%%%%%%%%%%%%%%%%%%%
%%%%%%%%%%%%%%%%%%%%%%%%%%%%%%%%%%%%%%%%%%%
%%%%%%%%%%%%%%%%%%%%%%%%%%%%%%%%%%%%%%%%%%%
%%%%%%%%%%%%%%%%%%%%%%%%%%%%%%%%%%%%%%%%%%%
%%%%%%%%%%%%%%%%%%%%%%%%%%%%%%%%%%%%%%%%%%%

\section{Distance-biregular graph with $\boldsymbol{c_2'=1}$ and $\boldsymbol{k'\ge 3}$}
\label{Eb}

Let $\G$ denote a $(Y,Y')$-distance-biregular graph with $D\ge 3$ and $k'\ge 3$. 
In this section we show that $\G$ is almost $2$-$Y$-homogeneous with $c_2'=1$ if and only if $c_i=1$ for every integer $i$ $(1\le i\le D-1)$.

\begin{theorem}
\label{Ep}
With reference to Notation~\ref{GN}, let $\G$ denote a $(Y,Y')$-distance-biregular graph with $D\ge 3$ and $k'\ge 3$. The following are equivalent.
\begin{enumerate}[label=(\roman*), font=\rm]
\item $\G$ is almost $2$-$Y$-homogeneous and $c_2=1$. 
\item $c_i=1$ for every integer $i$ $(1\le i\le D-1)$.
\end{enumerate}
\end{theorem}

\begin{proof}
(i)$\Ra$(ii) By \eqref{r5}, $c_2'=1$. To obtain a contradiction we assume that there exists $t$ $(2\le t\le D-2)$ such that $c_j=1$ for all $1\leq j\leq t$ and $c_{t+1}>1$. With this assumption it follows from Corollary~\ref{Fb} that $c'_1=\cdots=c'_t=1$. Moreover, since $c_{t+1}>c'_t=1$, \eqref{Et} yields $c'_{t+1}>c_t=1$. Since $\G$ is almost $2$-$Y$-homogeneous and $c'_2=1$, by Corollary~\ref{El} $\Delta_t=0$, which implies that we have either $(b_{t-1}-1)(c_{t+1}-1)=0$ or $(b'_{t-1}-1)(c'_{t+1}-1)=0$, depending on the parity of $t$. Since both $c_{t+1}>1$ and $c'_{t+1}>1$ we have either $b_{t-1}=1$ or $b'_{t-1}=1$.

{\sc Case 1.} Suppose $t$ is even, and with it $b_{t-1}=1$. Lemma~\ref{4w} yields $b'_t=1$. Since $c'_t=1$ we get $k'=b'_t+c'_t=2$, a contradiction.

{\sc Case 2.} Suppose $t$ is odd, and with it $b'_{t-1}=1$. Lemma~\ref{4w} yields $b_t=1$. Since $c_t=1$ we get $k'=b_t+c_t=2$, a contradiction.

Since both cases contradict $k'>2$, the result follows.

\medskip
(ii)$\Ra$(i) Assume that $c_i=1$ for every $i$ $(1 \leq i \leq D-1)$. Then, Corollary~\ref{Fb} implies $c'_i=1$ for every $i$ $(1 \leq i \leq D-1)$. Thus $\Delta_i=0$ for every $i$ $(1 \leq i \leq D-2)$. The result now follows from Corollary~\ref{El}.
\end{proof}

%%%%%%%%%%%%%%%%%%%%%%%%%%%%%%%%%%%%%%%%%%%
%%%%%%%%%%%%%%%%%%%%%%%%%%%%%%%%%%%%%%%%%%%
%%%%%%%%%%%%%%%%%%%%%%%%%%%%%%%%%%%%%%%%%%%
%%%%%%%%%%%%%%%%%%%%%%%%%%%%%%%%%%%%%%%%%%%

%%%%%%%%%%%%%%%%%%%%%%%%%%%%%%%%%%%%%%%%%%%
%%%%%%%%%%%%%%%%%%%%%%%%%%%%%%%%%%%%%%%%%%%
%%%%%%%%%%%%%%%%%%%%%%%%%%%%%%%%%%%%%%%%%%%
%%%%%%%%%%%%%%%%%%%%%%%%%%%%%%%%%%%%%%%%%%%
%%%%%%%%%%%%%%%%%%%%%%%%%%%%%%%%%%%%%%%%%%%
%%%%%%%%%%%%%%%%%%%%%%%%%%%%%%%%%%%%%%%%%%%
%%%%%%%%%%%%%%%%%%%%%%%%%%%%%%%%%%%%%%%%%%%
%%%%%%%%%%%%%%%%%%%%%%%%%%%%%%%%%%%%%%%%%%%
%%%%%%%%%%%%%%%%%%%%%%%%%%%%%%%%%%%%%%%%%%%
%%%%%%%%%%%%%%%%%%%%%%%%%%%%%%%%%%%%%%%%%%%
%%%%%%%%%%%%%%%%%%%%%%%%%%%%%%%%%%%%%%%%%%%
%%%%%%%%%%%%%%%%%%%%%%%%%%%%%%%%%%%%%%%%%%%

\section{Distance-biregular graphs with $\boldsymbol{D=3}$}
\label{Fc}

Note that any $(Y,Y')$-distance-biregular graph with $D=3$ is almost $2$-$Y$-homogeneous, by definition. In this section we show that a $(Y,Y')$-distance-biregular graph with $D=3$ is $2$-$Y$-homogeneous if and only if $\left|\G_2(x)\right|=\deg(x)$. Note that $k'=2$ automatically yields $\left|\G_2(x)\right|=\deg(x)$ (any by Proposition~\ref{km} such graph is $2$-$Y$-homogeneous). In the next theorem (Theorem~\ref{ge}) we consider the case when $k'\ge 3$.

\begin{theorem}
\label{ge}
With reference to Notation~\ref{GN}, let $\G$ denote a $(Y,Y')$-distance-biregular graph $\G$ with $k'\ge 3$ and $D=3$. The following are equivalent.
\begin{enumerate}[label=(\roman*), font=\rm] 
\item $\Delta_2=0$.
\item $\G$ is $2$-$Y$-homogeneous.
\item $|\G_2(x)|=\deg(x)$.
\end{enumerate}
\end{theorem}

\begin{proof}
First note that, if $D$ is odd then $D'\ge D$. Considering the intersection diagram of rank $0$ (see Figure~\ref{01}) we have $|\G_2(x)|=\frac{b_0b_1}{c_2}$. Now, it is not hard to see that $|\G_2(x)|=\deg(x)$ holds if and only if $b_1=c_2$.

\smallskip
(i)$\Ra$(ii) Since $D=3$ is odd we have $D'\ge D$, and with it $\min\{D-1,D'-1\}=2$. The claim now follows immediately from Corollary~\ref{Ek}.

\smallskip
(ii)$\Ra$(iii) Assume that $\G$ is $2$-$Y$-homogeneous. By Corollary \ref{Ek}, the scalar $\Delta_2=0$, and since $D=3$, we have $c_3-1=b_1$. This yields
\begin{equation}
\label{Dt}
b_1(b_1-1)=p^2_{2,2}(c'_2-1).
\end{equation}
Suppose first that $c'_2=1$. Equation \eqref{r5} yields $c_2=1$. Moreover, the right-hand side of \eqref{Dt} is equal to $0$, and with it $b_1=1$. This shows that $b_1= c_2=1$ and the result follows. Assume now that $c'_2\ge 2$. Note that $k-1=b_1'=b_0-1$. Now from \eqref{r5} and \eqref{Dt} we have 
\begin{equation}
\label{Dv}
p^2_{2,2}=\frac{b_1(b_1-1)}{c'_2-1}=\frac{(b_0-1)(b_1-1)}{c_2-1}.
\end{equation}
On the other hand, since for any $x\in Y$ and $y\in\G_2(x)$, $p^2_{2,2}=|\G_{2,2}(x,y)|$ and $|\G_{2,4}(x,y)|=0$, by \eqref{Ej} and \eqref{id0} this yields
\begin{equation}
\label{Dw}
p^2_{2,2}=|\G_2(x)|-1=\frac{b_0b_1-c_2}{c_2}.
\end{equation}
By \eqref{Dv} and \eqref{Dw} we have $c_2(b_0-1)(b_1-1)=(b_0b_1-c_2)(c_2-1)$. Replacing $b_0$ by $b_2+c_2$ in the last equality, after simplification we get $b_2(b_1-c_2)=0$, which yields $c_2=b_1$. Hence, $|\G_2(x)|=\deg(x)$.

\smallskip
(iii)$\Ra$(i) Assume that $|\G_2(x)|=\deg(x)$. Note that $b_1'=b_0-1$ and so
$$
b_1'=\deg(x)-1=|\G_2(x)|-1=p^2_{2,2}.
$$
Now, by \eqref{r5} we have
\begin{equation}
\label{Dx}
b_1(c_2-1)=p^2_{2,2}(c'_2-1).
\end{equation}
After replacing $c_3$ by $b_1+c_1$ in the definition of $\Delta_2$, and applying \eqref{Dx} we have $\Delta_2=b_1(b_1-c_2)$. Since $b_1=c_2$ we get $\Delta_2=0$. The result now follows from Proposition~\ref{Dm}.
\end{proof}

\begin{corollary}
\label{Dy}
With reference to Notation~\ref{GN}, a $(Y,Y')$-distance-biregular graph with $D=3$ is $2$-$Y$-homogeneous if and only if $\left|\G_2(x)\right| =\deg(x)$. Moreover, if $\G$ is a $2$-$Y$-homogeneous distance-biregular graph with $D=3$, then the intersection array of the colour class $Y$ is of the following type:
$$
(k,c,k-c; 1, c, c+1).
$$
\end{corollary}

\begin{proof}
The equivalence follows immediately from Proposition~\ref{km} and Theorem~\ref{ge}. For the intersection array, note that $|\G_2(x)|=\deg(x)$ holds if and only if $b_1=c_2$. On the other hand $c_3=k'=b_1+c_1=c_2+1$, and the result follows.
\end{proof}

%%%%%%%%%%%%%%%%%%%%%%%%%%%%%%%%%%%%%%%%%%%
%%%%%%%%%%%%%%%%%%%%%%%%%%%%%%%%%%%%%%%%%%%
%%%%%%%%%%%%%%%%%%%%%%%%%%%%%%%%%%%%%%%%%%%
%%%%%%%%%%%%%%%%%%%%%%%%%%%%%%%%%%%%%%%%%%%

%%%%%%%%%%%%%%%%%%%%%%%%%%%%%%%%%%%%%%%%%%%
%%%%%%%%%%%%%%%%%%%%%%%%%%%%%%%%%%%%%%%%%%%
%%%%%%%%%%%%%%%%%%%%%%%%%%%%%%%%%%%%%%%%%%%
%%%%%%%%%%%%%%%%%%%%%%%%%%%%%%%%%%%%%%%%%%%
%%%%%%%%%%%%%%%%%%%%%%%%%%%%%%%%%%%%%%%%%%%
%%%%%%%%%%%%%%%%%%%%%%%%%%%%%%%%%%%%%%%%%%%
%%%%%%%%%%%%%%%%%%%%%%%%%%%%%%%%%%%%%%%%%%%
%%%%%%%%%%%%%%%%%%%%%%%%%%%%%%%%%%%%%%%%%%%
%%%%%%%%%%%%%%%%%%%%%%%%%%%%%%%%%%%%%%%%%%%
%%%%%%%%%%%%%%%%%%%%%%%%%%%%%%%%%%%%%%%%%%%
%%%%%%%%%%%%%%%%%%%%%%%%%%%%%%%%%%%%%%%%%%%
%%%%%%%%%%%%%%%%%%%%%%%%%%%%%%%%%%%%%%%%%%%

\section{Distance-biregular graphs with $\boldsymbol{D=4}$ and $\boldsymbol{D=5}$}
\label{Fe}

In this section we give possible types for the intersection array of a $2$-$Y$-homogeneous $(Y,Y')$-distance-biregular graph with $D=4$ and $D=5$, written in terms of three parameters.

\begin{lemma}
\label{Ff}
With reference to Notation~\ref{GN}, let $\G$ denote a $(Y,Y')$-distance-biregular graph with $D=4$. If $\G$ is $2$-$Y$-homogeneous with $c_2'\ge 2$ then the intersection array of the colour class $Y$ is of the following type:
$$
(k,k'-1,k-c, k'-1-\frac{c(c'-1)}{\gamma}; 1, c, \frac{c(c'-1)}{\gamma}+1, k)
$$
for some positive integers $k$, $k'$ and $c$, where $k>c\ge 2$, $k'>2$, $c'=\frac{(k'-1)(c-1)}{k-1}+1$ and $\gamma=\frac{(c-1)(c'-2)}{k'-2}+1$.
\end{lemma}

\begin{proof}
By assumption $c_2'\ge 2$, which yields $k'\ge 3$. Equation \eqref{r5}, Lemma~\ref{4w}(i) and Lemma~\ref{gH}(i)(ii) yield $c_2\ge 2$ and $\gamma_2\ge 1$. Note that $D=4$ implies $c_4=k$, and by \eqref{r5} we have $c_2'-1=\frac{(k'-1)(c_2-1)}{k-1}$. The result follows from Lemma~\ref{gH}(i) and Lemma~\ref{vH}(i).
\end{proof}

\begin{lemma}
\label{Fm}
With reference to Notation~\ref{GN}, let $\G$ denote a $(Y,Y')$-distance-biregular graph with $D=5$. If $\G$ is $2$-$Y$-homogeneous with $c_2'\ge 2$ then the intersection array of the colour class $Y$ is of the following type:
$$
(k,k'-1,k-c, 1+\frac{c(c'-1)}{\gamma}, b_4; 1, c, k'-1-\frac{c(c'-1)}{\gamma}, c_4, k')
$$
for some positive integers $k$, $k'$ and $c$, where $k>c\ge 2$, $k'>2$, $c'=\frac{(k'-1)(c-1)}{k-1}+1$, $\gamma=\frac{(c-1)(c'-2)}{k'-2}+1$, $c_4=\frac{k(k'-1)-\frac{c(b_3-1)(k-1)}{c-1}}{k'-1-\frac{c(c'-1)}{\gamma}}$ and $b_4=k-c_4$.
\end{lemma}

\begin{proof}
The result for the first three intersection numbers follow immediately from the proof of Lemma~\ref{Ff}. It is only left to compute $c_4$. 
First note that $D\ge 5$ yields $D'\ge 5$. From the definition of $\Delta_i$ and Corollary~\ref{Ek}, $p^4_{24}=\frac{(b_3-1)(k'-1)}{c_2'-1}$. Now from \eqref{Ej}, \eqref{au} and \eqref{Di}, we have $kb_1-c_3c_4=\frac{c_2(b_3-1)(k-1)}{c_2-1}$, and the result follows.
\end{proof}

\bigskip
\noindent
{\sc{Proof of Theorem~\ref{rH}.}} If $c_2\ge 2$, then from \eqref{r5} we have $b_1=k'-1=\frac{(c_2'-1)(k-1)}{c_2-1}$. Just for the moment assume that $c_2'=2$. Then by Lemmas~\ref{vH}(i) and \ref{gH}(i), $\gamma_2=1$, which together with Lemma~\ref{vH}(ii) yield $c_3=c_2+1$. The rest of theorem follows immediately from Proposition~\ref{km}, Theorems~\ref{hH}, \ref{Db}, \ref{Ep}, Corollary~\ref{Dy} and Lemmas~\ref{Ff}, \ref{Fm}.

%%%%%%%%%%%%%%%%%%%%%%%%%%%%%%%%%%%%%%%%%%%
%%%%%%%%%%%%%%%%%%%%%%%%%%%%%%%%%%%%%%%%%%%
%%%%%%%%%%%%%%%%%%%%%%%%%%%%%%%%%%%%%%%%%%%
%%%%%%%%%%%%%%%%%%%%%%%%%%%%%%%%%%%%%%%%%%%

%%%%%%%%%%%%%%%%%%%%%%%%%%%%%%%%%%%%%%%%%%%
%%%%%%%%%%%%%%%%%%%%%%%%%%%%%%%%%%%%%%%%%%%
%%%%%%%%%%%%%%%%%%%%%%%%%%%%%%%%%%%%%%%%%%%
%%%%%%%%%%%%%%%%%%%%%%%%%%%%%%%%%%%%%%%%%%%
%%%%%%%%%%%%%%%%%%%%%%%%%%%%%%%%%%%%%%%%%%%
%%%%%%%%%%%%%%%%%%%%%%%%%%%%%%%%%%%%%%%%%%%
%%%%%%%%%%%%%%%%%%%%%%%%%%%%%%%%%%%%%%%%%%%
%%%%%%%%%%%%%%%%%%%%%%%%%%%%%%%%%%%%%%%%%%%
%%%%%%%%%%%%%%%%%%%%%%%%%%%%%%%%%%%%%%%%%%%
%%%%%%%%%%%%%%%%%%%%%%%%%%%%%%%%%%%%%%%%%%%
%%%%%%%%%%%%%%%%%%%%%%%%%%%%%%%%%%%%%%%%%%%
%%%%%%%%%%%%%%%%%%%%%%%%%%%%%%%%%%%%%%%%%%%

\section{Simple examples and suggestions for further research}
\label{Fh}

In this section we give simple examples of (almost) $2$-$Y$-homogeneous graphs, and present some open problems for future research.

\begin{example}{\rm
\label{Hj}
Let $\Gamma^\circ=(X^\circ, \mathcal{R}^\circ)$ denote the Petersen graph. The Petersen graph is a $(3,5)$-cage with diameter $2$ and odd girth. The subdivision graph of the Petersen graph is $2$-$X^\circ$-homogeneous but it is not almost  $2$-$\mathcal{R}^\circ$-homogeneous (see also Figure \ref{03}).
}\end{example}

\begin{example}{\rm
\label{Hk}
The Heawood graph $\Gamma^\circ=(X^\circ, \mathcal{R}^\circ)$ is a $(3,6)$-cage with diameter $3$ and even girth. The subdivision graph of the Heawood graph is $2$-$X^\circ$-homogeneous, it is almost $2$-$\mathcal{R}^\circ$-homogeneous but it is not $2$-$\mathcal{R}^\circ$-homogeneous.
}\end{example}

\begin{example}{\rm
\label{Fl}
Consider the set $\mathcal{P}=\left\lbrace 1,2,3,4,5,6,7,8\right\rbrace $ and let $\mathcal{B}$ be the collection of the following nonempty subsets of $\mathcal{P}$: 
	\begin{eqnarray}
	&&\left\lbrace 1,3,7,8\right\rbrace, 
	\left\lbrace 1,2,4,8\right\rbrace,
	\left\lbrace 2,3,5,8\right\rbrace, 
	\left\lbrace 3,4,6,8\right\rbrace,
	\left\lbrace 4,5,7,8\right\rbrace ,
	\left\lbrace 1,5,6,8\right\rbrace,
	\left\lbrace 2,6,7,8\right\rbrace, \nonumber \\ 	
	&&\left\lbrace 1,2,3,6 \right\rbrace, 	
	\left\lbrace 1,2,5,7 \right\rbrace, 
	\left\lbrace 1,3,4,5 \right\rbrace,
	\left\lbrace 1,4,6,7 \right\rbrace,
	\left\lbrace 2,3,4,7 \right\rbrace,
	\left\lbrace 2,4,5,6\right\rbrace,
	\left\lbrace 3,5,6,7 \right\rbrace.	\nonumber 
	\end{eqnarray} 
Let $\G$ denote the bipartite graph with color classes $\mathcal{P}$ and $\mathcal{B}$ where $\{p,B\}$ is an edge of $\G$ ($p\in\P$, $B\in\B$) if and only if $p\in B$. It is easy to check that $\G$ is distance-biregular. Moreover, every vertex in $\mathcal{P}$ has eccentricity equal to $3$ and intersection array $(7,3,4;1,3,4)$. Pick $p\in \mathcal{P}$. By \eqref{Ej}, $|\G_2(p)|=7=\deg(p)$. Corollary~\ref{Dy} yields that $\G$ is $2$-$\mathcal{P}$-homogeneous.
}\end{example}

\begin{example}{\rm
For a natural number $n$,  an {\it $n$-grid} is an incidence structure $\mathcal{S}= (\mathcal{P},\mathcal{B},\mathcal{I})$ where  $\mathcal{P}=\left\lbrace x_{ij} | \; 0 \leq i, j \leq n \right\rbrace $ and $\mathcal{B}=\left\lbrace L_0, \cdots, L_n, M_0, \cdots, M_n \right\rbrace $ such that $x_{ij}$ lies on $L_k$ if and only if $i=k$ and  $x_{ij}$ lies on $M_k$ if and only if $j=k$. Generalized quadrangles were introduced by J. Tits \cite{tits}. We follow \cite{thas} for the standard concepts in generalized quadrangles. It is easy to see that an $n$-grid is a generalized quadrangle with parameters $s=n$ and $t=1$.  For $n\geq 2$, let $\G(n)=(X, \mathcal{R})$ denote the incidence graph of the $n$-grid (that is, the graph with $X=\mathcal{P} \cup \mathcal{B}$ where two vertices $ p\in  \mathcal{P},  \ell \in  \mathcal{B}$ are adjacent if and only if $p$ lies on  $\ell$). As the $n$-grid is a generalized quadrangle, it is well-known that $\G(n)$ is characterized by being a connected, bipartite graph with diameter $4$ and girth $8$. In addition, we observe $\G(n)$ is distance-biregular. The intersection array of every vertex $p \in \mathcal{P}$ is $(2, n, 1, n; 1, 1, 1, 2)$ while every vertex $\ell \in \mathcal{B}$ has intersection array $(n+1, 1, n, n; 1, 1, 1, n+1)$.  Since the numbers $c_i=c'_i=1 \;(1\leq i \leq 3)$ it follows from Theorem~\ref{Ep} that  $\G(n)$ is both almost $2$-$\mathcal{P}$-homogeneous and almost $2$-$\mathcal{B}$-homogeneous.
}\end{example}

In Section~\ref{D2} it is proven that if $\G^\circ$ is a $(\kappa,g)$-cage graph with vertex set $X^\circ$ then the subdivision graph of $\G^\circ$ is $2$-$X^\circ$-homogeneous. In addition, in Proposition~\ref{Fj} we show that, if $g$ is even then the subdivision graph of $\G^\circ$ is almost $2$-$\R^\circ$-homogeneous.

\begin{proposition}
\label{Fj}
Let $\Gamma^\circ=(X^\circ, \mathcal{R}^\circ)$ denote a $(\kappa,g)$-cage graph with $\kappa\ge 3$ and $g\ge 3$ and let $\Gamma$ denote the subdivision graph of $\Gamma^\circ$. If $g$ is even then $\G$ is almost $2$-$\R^\circ$-homogeneous.
\end{proposition}

\begin{proof}
The subdivision graph $\Gamma=S(\G^\circ)$ with vertex set $X=Y\cup Y'$ is distance-biregular with bipartite parts $Y=X^\circ$ and $Y'=\mathcal{R}^\circ$ (see Theorem~\ref{Gp}). Moreover, their intersection numbers depend on the parity of the girth $g$ of $\G^\circ$. Pick $x\in Y$ and $x'\in Y'$. Note $k=\deg(x)=\kappa$ and $k'=\deg(x')=2$. Assume that $g$ is even. Since $x$ and $x'$ have the same eccentricity  $\epsilon(x)=\epsilon(x')=2d$, we can compute $\Delta_i(\mathcal{R}^\circ)$ for every integer $i \; (1 \leq i \leq 2d-1)$. Recall that $c_2=1$ and that $\deg(x)=\kappa>2$. Then, by Theorem~\ref{Gp} we have $\Delta_i=\Delta_i(\mathcal{R}^\circ)=0$ for all $i$ $(1\leq i \leq 2d-2)$. Hence, by Corollary~\ref{El}, the subdivision graph $\G$ is almost $2$-$\mathcal{R}^\circ$-homogeneous. 
\end{proof}

%%\section{Further directions}
%%\label{Fi}

The problem which is beyond our reach is to find an algorithm for constructing a distance-biregular graph from its two intersection arrays (under the assumption that such graph exists). As far as we know, this is not known in the literature and we kindly ask the reader to contact us about any results in this direction. To explain what we want here, let us give a concrete example: we know there exists a distance-biregular graph with intersection arrays $(3,1,2,1,2;\,1,1,1,1,2)$ and $(2,2,1,2,1,1;\,1,1,1,1,2,2)$. We would like to find an algorithm which as input has these two intersection arrays, and which will as output gives us the adjacency matrix of a corresponding graph (note that, in general case, it can happen that such a graph is not unique). Moreover, the same problem can be set up for a bipartite distance-regular graph but the problem is not easier. 

\begin{problem}
Let $\G$ denote a distance-biregular graph. Find an algorithm, if possible, to construct $\G$ from its two intersection arrays. 
\end{problem}

Now, let $\G$ denote a distance-biregular graph with color partitions $(Y,Y')$ and $k'\ge 3$. Recall that in Theorem~\ref{ge} we showed that if $D=3$ then $\G$ is $2$-$Y$-homogeneous if and only if $|\G_2(x)|=\deg(x)$, and we have examples of such graphs (see Example~\ref{Fl}). Thus, the following problem is also interesting.

\begin{problem}
Let $\G$ denote a distance-biregular graph with color partitions $Y$ and $Y'$ such that the following {\rm (i)--(iii)} hold:
\begin{enumerate}[label=(\roman*), font=\rm]
\item $D\ge 4$ and $k'\ge 3$,
\item $\G$ is $2$-$Y$-homogeneous,
\item $|\G_2(x)|> \deg(x)$ for every $x\in Y$.
\end{enumerate}
Prove or disprove that such a graph $\G$ exists. 
\end{problem}

Let's mention two more open problems 
%%that are interesting, and 
that we didn't manage to solve in this paper: (a)~find the intersection array of $2$-$Y$-homogeneus distance-biregular graph for the case when $c_2'=2$ (see Theorem~\ref{rH}(ii)); and (b)~prove (or disprove) that the two claims from Remark~\ref{Ey} are equivalent.

\section{Acknowledgments}

This work is supported in part by the Slovenian Research Agency (research program P1-0285, research project J1-2451 and Young Researchers Grant).

{\small
\bibliographystyle{references}
\bibliography{almost2homogDBRG}

\begin{thebibliography}{10}
\expandafter\ifx\csname urlstyle\endcsname\relax
  \providecommand{\doi}[1]{doi:\discretionary{}{}{}#1}\else
  \providecommand{\doi}{doi:\discretionary{}{}{}\begingroup
  \urlstyle{rm}\Url}\fi

\bibitem{NBs}
N.~Biggs, The symmetry of line graphs, \emph{Utilitas Math.} \textbf{5} (1974),
  113--121.

\bibitem{NB}
N.~Biggs, \emph{Algebraic graph theory}, Cambridge Mathematical Library,
  Cambridge University Press, Cambridge, 2nd edition, 1993.

\bibitem{BN}
N.~L. Biggs, Potential theory on distance-regular graphs, \emph{Combin. Probab.
  Comput.} \textbf{2} (1993), 243--255, \doi{10.1017/S096354830000064X}.

\bibitem{BCN}
A.~E. Brouwer, A.~M. Cohen and A.~Neumaier, \emph{Distance-regular graphs},
  volume~18 of \emph{Ergebnisse der Mathematik und ihrer Grenzgebiete (3)
  [Results in Mathematics and Related Areas (3)]}, Springer-Verlag, Berlin,
  1989, \doi{10.1007/978-3-642-74341-2}.

\bibitem{BC}
B.~Curtin, {$2$}-homogeneous bipartite distance-regular graphs, \emph{Discrete
  Math.} \textbf{187} (1998), 39--70, \doi{10.1016/S0012-365X(97)00226-4}.

\bibitem{CZ}
Z.~Cvetkovski, \emph{Inequalities}, Springer, Heidelberg, 2012,
  \doi{10.1007/978-3-642-23792-8}, theorems, techniques and selected problems.

\bibitem{DKT}
E.~{\tiny\phantom{d}\hspace{-2mm}}van Dam, J.~H. Koolen and H.~Tanaka,
  \emph{Distance-regular graphs}, Dynamic Surveys, Electron. J. Combin., 2016,
  \url{http://www.combinatorics.org/ojs/index.php/eljc/article/view/DS22/pdf}.

\bibitem{DRM}
R.~M. Damerell, On {M}oore graphs, \emph{Proc. Cambridge Philos. Soc.}
  \textbf{74} (1973), 227--236, \doi{10.1017/s0305004100048015}.

\bibitem{DC}
C.~Delorme, Distance biregular bipartite graphs, \emph{European J. Combin.}
  \textbf{15} (1994), 223--238, \doi{10.1006/eujc.1994.1024}.

\bibitem{EMJS}
J.~A. Ellis-Monaghan and I.~Sarmiento, Distance hereditary graphs and the
  interlace polynomial, \emph{Combin. Probab. Comput.} \textbf{16} (2007),
  947--973, \doi{10.1017/S0963548307008723}.

\bibitem{EJ}
G.~Exoo and R.~Jajcay, Recursive constructions of small regular graphs of given
  degree and girth, \emph{Discrete Math.} \textbf{312} (2012), 2612--2619,
  \doi{10.1016/j.disc.2011.10.021}.

\bibitem{EJS}
G.~Exoo, R.~Jajcay and J.~\v{S}ir\'{a}\v{n}, Cayley cages, \emph{J. Algebraic
  Combin.} \textbf{38} (2013), 209--224, \doi{10.1007/s10801-012-0400-2}.

\bibitem{FM}
B.~Fern\'{a}ndez and {\v{S}}.~Miklavi\v{c}, On the {T}erwilliger algebra of
  distance-biregular graphs, \emph{Linear Algebra Appl.} \textbf{597} (2020),
  18--32, \doi{10.1016/j.laa.2020.03.016}.

\bibitem{FMe}
B.~Fern\'{a}ndez and {\v{S}}.~Miklavi\v{c}, On bipartite graphs with exactly
  one irreducible {$T$}-module with endpoint 1, which is thin, \emph{European
  J. Combin.} \textbf{97} (2021), Paper No. 103387, 15,
  \doi{10.1016/j.ejc.2021.103387}.

\bibitem{FKR}
M.~Ferrara, Y.~Kohayakawa and V.~R\"{o}dl, Distance graphs on the integers,
  \emph{Combin. Probab. Comput.} \textbf{14} (2005), 107--131,
  \doi{10.1017/S0963548304006637}.

\bibitem{FMA}
M.~A. Fiol, Some spectral characterizations of strongly distance-regular
  graphs, \emph{Combin. Probab. Comput.} \textbf{10} (2001), 127--135,
  \doi{10.1017/S0963548301004564}.

\bibitem{Fp}
M.~A. Fiol, Pseudo-distance-regularized graphs are distance-regular or
  distance-biregular, \emph{Linear Algebra Appl.} \textbf{437} (2012),
  2973--2977, \doi{10.1016/j.laa.2012.07.019}.

\bibitem{Fs}
M.~A. Fiol, The spectral excess theorem for distance-biregular graphs,
  \emph{Electron. J. Combin.} \textbf{20} (2013), Paper 21, 10,
  \doi{10.37236/3305}.

\bibitem{GST}
C.~D. Godsil and J.~Shawe-Taylor, Distance-regularised graphs are
  distance-regular or distance-biregular, \emph{J. Combin. Theory Ser. B}
  \textbf{43} (1987), 14--24, \doi{10.1016/0095-8956(87)90027-X}.

\bibitem{HK}
F.~Harary and P.~Kov\'{a}cs, Regular graphs with given girth pair, \emph{J.
  Graph Theory} \textbf{7} (1983), 209--218, \doi{10.1002/jgt.3190070210}.

\bibitem{KTKR}
T.~Kaiser and R.~J. Kang, The distance-{$t$} chromatic index of graphs,
  \emph{Combin. Probab. Comput.} \textbf{23} (2014), 90--101,
  \doi{10.1017/S0963548313000473}.

\bibitem{KZ}
A.~Kotzig and B.~Zelinka, Regular graphs, each edge of which belongs to exactly
  one {$s$}-gon, \emph{Mat. \v{C}asopis Sloven. Akad. Vied} \textbf{20} (1970),
  181--184.

\bibitem{MM1}
M.~S. MacLean and {\v{S}}.~Miklavi\v{c}, On bipartite distance-regular graphs
  with exactly one non-thin {$T$}-module with endpoint two, \emph{European J.
  Combin.} \textbf{64} (2017), 125--137, \doi{10.1016/j.ejc.2017.04.004}.

\bibitem{MM2}
M.~S. MacLean and {\v{S}}.~Miklavi\v{c}, On bipartite distance-regular graphs
  with exactly two irreducible {T}-modules with endpoint two, \emph{Linear
  Algebra Appl.} \textbf{515} (2017), 275--297,
  \doi{10.1016/j.laa.2016.11.021}.

\bibitem{MM4}
M.~S. MacLean and {\v{S}}.~Miklavi\v{c}, Bipartite distance-regular graphs and
  taut pairs of pseudo primitive idempotents, \emph{Algebr. Comb.} \textbf{2}
  (2019), 499--520, \doi{10.5802/alco.51}.

\bibitem{MM3}
M.~S. MacLean and {\v{S}}.~Miklavi\v{c}, On a certain class of $1$-thin
  distance-regular graphs, \emph{Ars Math. Contemp.} \textbf{18} (2020),
  187--210, \doi{10.26493/1855-3974.2193.0b0}.

\bibitem{MMP1}
M.~S. MacLean, {\v{S}}.~Miklavi\v{c} and S.~Penji\'{c}, On the {T}erwilliger
  algebra of bipartite distance-regular graphs with {$\Delta_2=0$} and
  {$c_2=1$}, \emph{Linear Algebra Appl.} \textbf{496} (2016), 307--330,
  \doi{10.1016/j.laa.2016.01.040}.

\bibitem{MMP2}
M.~S. MacLean, {\v{S}}.~Miklavi\v{c} and S.~Penji\'{c}, An {$A$}-invariant
  subspace for bipartite distance-regular graphs with exactly two irreducible
  {$T$}-modules with endpoint 2, both thin, \emph{J. Algebraic Combin.}
  \textbf{48} (2018), 511--548, \doi{10.1007/s10801-017-0798-7}.

\bibitem{MmTp}
M.~S. MacLean and P.~Terwilliger, The subconstituent algebra of a bipartite
  distance-regular graph; thin modules with endpoint two, \emph{Discrete Math.}
  \textbf{308} (2008), 1230--1259, \doi{10.1016/j.disc.2007.03.071}.

\bibitem{AAM}
A.~A. Makhnev, Moore graph with parameters $(3250,57,0,1)$ does not exist,
  2020, \url{https://arxiv.org/abs/2010.13443}.

\bibitem{MRRA}
R.~R. Martin and A.~W.~N. Riasanovsky, On the edit distance function of the
  random graph, \emph{Combin. Probab. Comput.}  (2021), 1--23,
  \doi{10.1017/S0963548321000353}.

\bibitem{WMS}
W.~J. Martin, Scaffolds: a graph-theoretic tool for tensor computations related
  to {B}ose-{M}esner algebras, \emph{Linear Algebra Appl.} \textbf{619} (2021),
  50--106, \doi{10.1016/j.laa.2021.02.009}.

\bibitem{MP}
{\v{S}}.~Miklavi\v{c} and S.~Penji\'{c}, On the {T}erwilliger algebra of a
  certain family of bipartite distance-regular graphs with {$\Delta_2=0$},
  \emph{Art Discrete Appl. Math.} \textbf{3} (2020), Paper No. 2.04, 14,
  \doi{10.26493/2590-9770.1271.e54}.

\bibitem{MST}
B.~Mohar and J.~Shawe-Taylor, Distance-biregular graphs with {$2$}-valent
  vertices and distance-regular line graphs, \emph{J. Combin. Theory Ser. B}
  \textbf{38} (1985), 193--203, \doi{10.1016/0095-8956(85)90065-6}.

\bibitem{NP}
A.~Neumaier and S.~Penji{\'c}, A unified view of inequalities for
  distance-regular graphs, part {I}, \emph{J. Combin. Theory Ser. B}  (2020),
  \doi{10.1016/j.jctb.2020.09.015}.

\bibitem{NLV}
V.~Neumann-Lara, {$k$}-{H}amiltonian graphs with given girth, in:
  \emph{Infinite and finite sets ({C}olloq., {K}eszthely, 1973; dedicated to
  {P}. {E}rd\H{o}s on his 60th birthday), {V}ol. {III}}, pp. 1133--1142.
  Colloq. Math. Soc. Jan\'{o}s Bolyai, Vol. 10, 1975.

\bibitem{NK}
K.~Nomura, Intersection diagrams of distance-biregular graphs, \emph{J. Combin.
  Theory Ser. B} \textbf{50} (1990), 214--221,
  \doi{10.1016/0095-8956(90)90076-C}.

\bibitem{NKsm}
K.~Nomura, Spin models on bipartite distance-regular graphs, \emph{J. Combin.
  Theory Ser. B} \textbf{64} (1995), 300--313, \doi{10.1006/jctb.1995.1037}.

\bibitem{thas}
S.~E. Payne and J.~A. Thas, \emph{Finite generalized quadrangles}, volume~9,
  European Mathematical Society, 2009.

\bibitem{PS}
S.~Penji{\'{c}}, On the {T}erwilliger algebra of bipartite distance-regular
  graphs with {$\Delta_2=0$} and {$c_2=2$}, \emph{Discrete Math.} \textbf{340}
  (2017), 452--466, \doi{10.1016/j.disc.2016.09.001}.

\bibitem{SP}
S.~Penji{\'c}, \emph{On the {T}erwilliger algebra of bipartite distance-regular
  graphs}, {U}niversity of {P}rimorska, 2019, thesis (Ph.D.),
  \url{http://osebje.famnit.upr.si/~penjic/research/}.

\bibitem{RO}
O.~E. Raz, A note on distinct distances, \emph{Combin. Probab. Comput.}
  \textbf{29} (2020), 650--663, \doi{10.1017/s096354832000022x}.

\bibitem{sachs}
H.~Sachs, Regular graphs with given girth and restricted circuits,
  \emph{Journal of the London Mathematical Society} \textbf{1} (1963),
  423--429.

\bibitem{SHg}
H.~Sachs, On regular graphs with given girth, in: \emph{Theory of {G}raphs and
  its {A}pplications ({P}roc. {S}ympos. {S}molenice, 1963)}, Publ. House
  Czechoslovak Acad. Sci., Prague, 1964 pp. 91--97.

\bibitem{SH}
H.~Suzuki, On distance-biregular graphs of girth divisible by four,
  \emph{Graphs Combin.} \textbf{10} (1994), 61--65, \doi{10.1007/BF01202471}.

\bibitem{tits}
J.~Tits, Sur la trialit{\'e} et certains groupes qui s' en d{\'e}duisent,
  \emph{Publications Math{\'e}matiques de l'IH{\'E}S} \textbf{2} (1959),
  13--60.

\end{thebibliography}
}

\end{document}